\documentclass[english]{amsart}
\RequirePackage[OT1]{fontenc}
\RequirePackage[inline]{enumitem}
\RequirePackage{amsthm, amsmath, amssymb, amsfonts}
\RequirePackage[utf8]{inputenc}
\RequirePackage[numbers]{natbib}
\RequirePackage{hyperref}

\usepackage[T1]{fontenc}
\usepackage[inline]{enumitem}
\usepackage[utf8]{inputenc}
\usepackage{amssymb, amsthm, amsmath, amsfonts}
\usepackage{comment}
\usepackage{url}

\makeatletter
\numberwithin{equation}{section}
\numberwithin{figure}{section}

\theoremstyle{plain}
\newtheorem{theorem}{Theorem}[section]
\newtheorem*{theorem-non}{Theorem}
\newtheorem{proposition}[theorem]{Proposition}
\newtheorem{definition}[theorem]{Definition}
\newtheorem{lem}[theorem]{Lemma}
\newtheorem{cor}[theorem]{Corollary}

\theoremstyle{remark}
\newtheorem{remark}[theorem]{Remark}
\newtheorem{example}[theorem]{Example}

\newcommand{\Hom}{\textnormal{Hom}}

\newcommand{\Ker}{\textnormal{Ker}}

\newcommand{\ExpSig}{\textnormal{ExpSig}}

\newcommand{\II}{\mathcal{I}}

\newcommand{\CC}{\mathcal{C}}
\newcommand{\PP}{\mathcal{P}}
\newcommand{\DD}{\mathcal{D}}
\newcommand{\EE}{\mathcal{E}}

\newcommand{\UU}{\mathcal{U}}

\newcommand{\NN}{\mathcal{N}}
\newcommand{\KK}{\mathcal{K}}

\newcommand{\A}{\mathcal{A}}

\newcommand{\R}{\mathbb R}
\newcommand{\C}{\mathbb C}
\newcommand{\N}{\mathbb N}

\newcommand{\F}{\mathbb F}

\newcommand{\X}{\mathbf X}
\newcommand{\Y}{\mathbf Y}
\newcommand{\x}{\mathbf x}
\newcommand{\y}{\mathbf y}
\newcommand{\hh}{\mathbf h}

\newcommand{\pvar}{p\textnormal{-var}}
\newcommand{\pprimevar}{p'\textnormal{-var}}
\newcommand{\aHol}{\alpha\textnormal{-H{\"o}l}}
\newcommand{\aprimeHol}{\alpha'\textnormal{-H{\"o}l}}
\newcommand{\betaHol}{\beta\textnormal{-H{\"o}l}}
\newcommand{\pHol}{1/p\textnormal{-H{\"o}l}}

\newcommand{\EEE}[1]{\mathbb E \left[ #1 \right]}

\newcommand{\PPP}[1]{\mathbb{P} \left[ #1 \right]}

\newcommand{\spn}[1]{\textnormal{span}\left( #1 \right)}
\newcommand{\norm}[1]{\left|\left|#1 \right|\right|}

\newcommand{\normc}{\norm{\cdot}}

\newcommand{\1}{\mathbf 1}

\newcommand{\LLL}{\mathbf L}

\newcommand{\Lyons}{S}

\newcommand{\uu}{\mathfrak u}
\newcommand{\g}{\mathfrak g}

\newcommand{\gl}{\mathfrak {gl}}
\newcommand{\su}{\mathfrak{su}}
\newcommand{\symp}{\mathfrak{sp}}

\newcommand{\gen}[1]{\langle #1 \rangle}
\newcommand{\floor}[1]{\lfloor #1 \rfloor}

\def\eqd{\,{\buildrel \DD \over =}\,}

\def\convd{\,{\buildrel \DD \over \rightarrow}\,}

\makeatother

\usepackage{babel}

\begin{document}

\title[The char function for a random path]{Characteristic functions of measures on geometric rough paths}

\author{Ilya Chevyrev}
\address{I. Chevyrev,
Mathematical Institute,
University of Oxford,
Andrew Wiles Building,
Radcliffe Observatory Quarter,
Woodstock Road,
Oxford OX2 6GG,
United Kingdom}
\email{chevyrev@maths.ox.ac.uk}
\thanks{Supported by the University of Oxford Clarendon Fund Scholarship.}

\author{Terry Lyons}
\address{T. Lyons,
Oxford-Man Institute of Quantitative Finance,
University of Oxford,
Walton Well Road,
Oxford OX2 6ED,
United Kingdom}
\email{tlyons@maths.ox.ac.uk}
\thanks{Supported by ERC (Grant Agreement No.291244 Esig) and by the Oxford-Man Institute of Quantitative Finance.}





\begin{abstract}
We define a characteristic function for probability measures on the signatures of geometric rough paths. We determine sufficient conditions under which a random variable is uniquely determined by its expected signature, thus partially solving the analogue of the moments problem. We furthermore study analyticity properties of the characteristic function and prove a method of moments for weak convergence of random variables. We apply our results to signature arising from L{\'e}vy, Gaussian and Markovian rough paths.
\end{abstract}

\maketitle

\section{Introduction}

Paths serve as a natural description of an ordered progression of events and are abundant throughout mathematics. Furthermore, measures on paths are almost as common in nature as paths. Considering the flow of infinitesimal elements, one sees that any system involving rigid motions can be represented as a measure on paths; the same can be said of a gas or fluid flow. For this reason, the ability to characterize paths, and measures on them, becomes of value.

It was first shown by Chen~\cite{Chen58} that an irreducible piecewise regular continuous path in Euclidean space (which includes all paths that are smooth when parameterized at unit speed) may be faithfully represented, up to reparametrization, by the collection of its iterated integrals known as the signature.
The representation of a path through its signature has been recently explored in much greater detail due to its connection with rough paths theory~\cite{Lyons07}. The exact geometric equivalence of paths of bounded variation possessing the same signature was first described by Hambly and Lyons~\cite{Hambly10}, and recently extended to all geometric rough paths~\cite{Boedihardjo14}. Methods to recover information encoded by the signature have also been explored and, in general, pose a difficult problem~\cite{LyonsXu14}.

The signature may be viewed concretely as the universal solution to the exponential differential equation $d\Lyons(X)_t = \Lyons(X)_t \otimes dX_t$, and serves as the fully non-commutative analogue of the classical exponential function for points in $\R$. Its importance is further emphasized when one considers a general differential equation
\begin{equation}\label{eq_diffEq}
dY_t = M(Y_t)dX_t,
\end{equation}
since the solution $Y_t$ is invariant under reparametrizations of the driving signal $X_t$.
This relationship is most evident in the case of linear differential equations, where $X_t \in V$ and $Y_t \in W$ lie in Banach spaces, and $M : V \mapsto \LLL(W)$ is a continuous linear map. In this case the extension of $M$ to an algebra homomorphism $M : T(V) \mapsto \LLL(W)$, when applied to the signature of $X_t$, provides a series converging rapidly to the flow of~\eqref{eq_diffEq}~\cite{Lyons07}. In particular, when $M$ takes values in a Lie algebra, the flow of~\eqref{eq_diffEq} corresponds to the Cartan development of $X_t$ in the corresponding Lie group, thus naturally inducing a representation of the group of signatures.

In the case of a one-dimensional path $X_t$ in $\R$, the signature takes the simple form $(1, X_t-X_0, (X_t-X_0)^2/2!, \ldots)$. When $X_t$ is a random variable, the sequence of expectations $(1, \EEE{X_t-X_0}, \EEE{(X_t-X_0)^2/2!}, \ldots)$, whenever it exists, describes precisely the moments of $X_t-X_0$. Thus, for a stochastic process $X_t$, the expectations of its iterated integrals, termed the expected signature, naturally form the generalization of the moments of the process.

The expected signature has been exploited in high order approximation schemes~\cite{LittererLyons12} and is explicitly known for certain stochastic processes~\cite{FrizShekhar12, LyonsNi14}. Moreover, the fundamental property that every polynomial function on signatures may be realized as a linear functional implies that the expected signature distinguishes any two random variables of compact support~\cite{Fawcett03} and implicitly demonstrates the potential of the path signature in applications to numerical analysis and machine learning~\cite{Lyons14, LyonsVictoir}.

The moments of a random variable are of course closely related to the characteristic function $\phi_X(\lambda) = \EEE{e^{i \lambda X}}$. For a topological group $G$, a classical extension of the characteristic function to a $G$-valued random variable $X$ is $\phi_X(M) = \EEE{M(X)}$ where $M$ is a unitary representation of $G$~\cite{Heyer77}. Under suitable conditions, particularly the existence of sufficiently many unitary representations, $\phi_X$ uniquely determines the law of $X$.

This paper aims to study the characteristic function $\phi_X(M) = \EEE{M(X)}$ where $X$ is a random signature and $M$ is a unitary representation arising from a linear map $M : V \mapsto \uu$ into a unitary Lie algebra. Our main result asserts that $\phi_X$ uniquely determines every random variable $X$ and greatly extends the analogous result for the expected signature beyond the case of compact support.

We now briefly outline the structure of the paper. Section~\ref{sec_universalAlg} studies a universal topological algebra $E(V)$ in which we embed the group of signatures. Roughly speaking, the induced topology is such that a sequence of signatures converges if and only if the solution to~\eqref{eq_diffEq} converges for every continuous linear map $M: V \mapsto \LLL(W)$. In Section~\ref{sec_groupLike} we derive important properties of probability measures on the set $G(V)$ of group-like elements of $E(V)$. In Section~\ref{sec_reps} we study representations of $E(V)$. Our first main result is Theorem~\ref{thm_sepPointsUnitary}, which describes explicitly a family of representations of $E(\R^d)$ which preserves unitary elements and separates the points. An immediate consequence is that one is able to define a meaningful characteristic function for $G(\R^d)$-valued random variables (Corollary~\ref{cor_uniqueMeasure}). Though our results for uniqueness of random variables are restricted to the case $V=\R^d$, we mostly work in the general setting of Banach spaces and make precise whenever finite dimensionality is required.

In Section~\ref{sec_signature} we recall elements of rough paths theory and show that the signatures of geometric rough paths form a topological subgroup of $G(\R^d)$. In Section~\ref{sec_expectedSign} we describe applications of our results to stochastic rough paths, particularly in connection with the expected signature. We split Section~\ref{sec_expectedSign} into three parts.

In Section~\ref{subsec_momentsProb} we study the analogue of the moments problem.
Proposition~\ref{prop_infRadUnique} provides a general criterion under which a $G(\R^d)$-valued random variable is uniquely determined by its expected signature. In turn, Theorem~\ref{thm_r1Pos} provides a method to verify this criterion without explicit knowledge of the expected signature itself. We demonstrate applications of these results to the L{\'e}vy-Khintchine formula derived in~\cite{FrizShekhar12} and to families of Gaussian and Markovian rough paths studied in~\cite{CassLittererLyons13} and~\cite{CassOgrodnik15}.

In Section~\ref{subsec_analyticity} we study analyticity properties of the characteristic function. The main result is Theorem~\ref{thm_charFuncAnalytic} (and its Corollaries~\ref{cor_controlOnMp} and~\ref{cor_uniqueInPhi}), which provides a criterion to establish analyticity of the characteristic function and solve the moments problem within a restricted family of random variables. We demonstrate an application to Markovian rough paths stopped upon exiting a domain.

In Section~\ref{subsec_convMeas} we conclude with Theorem~\ref{thm_methodMoments}, which demonstrates a method of moments for weak convergence of $G(\R^d)$-valued random variables.

\subsection*{Acknowledgments}

The authors would like to thank Dr. Ni Hao and Prof. Peter Friz for numerous constructive discussions and suggestions, particularly during the Berlin-Oxford Meetings on Applied S.A. in 2013/2014. The authors would also like to thank Prof. Thierry L{\'e}vy for several helpful conversations, and Dr. Horatio Boedihardjo for valuable comments on an earlier draft.

\section{Universal locally \texorpdfstring{$m$}{m}-convex algebra}\label{sec_universalAlg}

Throughout the paper, all vector spaces are assumed real and all algebras are assumed unital. For topological vector spaces $V, W$, let $\LLL(V,W)$ be the space of continuous linear maps from $V$ to $W$, and denote $\LLL(V) = \LLL(V,V)$ and $V' = \LLL(V, \R)$. For terminology and basic properties of topological algebras we refer to~\cite{Mallios86}.

For a topological vector space $V$, a topological algebra $A$, and a topology on $T(V) = \bigoplus_{k \geq 0}V^{\otimes k}$, consider the statement:
\begin{equation}\label{eq_extensionCont}
\textnormal{For all $M \in \LLL(V, A)$, the extension $M: T(V) \mapsto A$ is continuous}.
\end{equation}

One may then topologize $T(V)$ by requiring that~\eqref{eq_extensionCont} holds for all topological algebras $A$ of a given category. In this paper we consider the category of locally $m$-convex algebras.

\begin{definition}
Let $V$ be a locally convex space. Let $E_a(V) = T(V)$ equipped with the coarsest topology such that~\eqref{eq_extensionCont} holds for all locally $m$-convex algebras $A$ (or equivalently, all normed algebras $A$). Denote by $E(V)$ the completion of $E_a(V)$.
\end{definition}

Thus for any normed algebra $A$, the set of continuous algebra homomorphisms $\Hom(E_a,A)$ is in bijection with $\LLL(V,A)$. For any $M \in \LLL(V,A)$ we shall usually denote by the same letter $M$ the corresponding element in $\Hom(E_a,A)$, but shall write $M_E \in \Hom(E_a,A)$ whenever a clear distinction is needed.

Though in most parts of the paper we shall assume that $V$ is normed, most results in this section are more easily understood for locally convex spaces and so unless stated otherwise, we only assume $V$ is locally convex.

In most of our notation, we shall drop the reference to $V$ when it is clear from the context. It holds that $E_a$ and $E$ are locally $m$-convex algebra (\cite{Mallios86} p.14, p.22). While most results in this section are stated for $E$, it is easy to verify which remain valid for $E_a$.

This method to obtain a universal topological algebra of a specific category is very natural, and we note that this construction is not new; the same construction (and essentially Proposition~\ref{prop_topEqual} below) appeared in~\cite{Cuntz97} in relation to cyclic cohomology, while analogous constructions were investigated for locally convex algebras with continuous multiplication in~\cite{Valqui01} and for commutative locally $m$-convex algebras (particularly in relation to nuclear spaces) in~\cite{Dineen81} Section~6.4.

\begin{remark}
If we start with $V$ as a general topological vector space, an easy verification shows that we arrive at the same space $E_a$ as when we equip $V$ with the finest locally convex topology coarser than its original.
\end{remark}

A family of semi-norms $\Psi$ on $V$ is called \emph{fundamental} if for every semi-norm $\xi$ on $V$, there exist $\gamma \in \Psi$ and $\varepsilon > 0$ such that $\varepsilon\xi \leq \gamma$ (note that by a semi-norm we always mean a continuous semi-norm).
For any collection of semi-norms $\Psi$ on $V$, define $\Psi^* = \{n \gamma \mid n \geq 1, \gamma \in \Psi\}$.

For semi-norms $\gamma,\xi$ on locally convex spaces $V,W$ respectively, let $\gamma\otimes\xi$ denote the projective semi-norm on $V\otimes W$.  Denote by $V \otimes_\pi W$ the projective tensor product and $V\widehat\otimes W$ its completion. For a normed space $F$, and $M \in \LLL(V,F)$ denote $\gamma(M) = \sup_{\gamma(v) = 1} \norm{M v}$ (possibly infinite).

Define the \emph{projective extension} of a semi-norm $\gamma$ on $V$ as the semi-norm $\exp(\gamma) = \sum_{k \geq 0} \gamma^{\otimes k}$ on $E_a$. Remark that $\exp(\gamma)$ is a sub-multiplicative semi-norm on $E_a$. Moreover for any normed algebra $A$, $M \in \LLL(V,A)$, and a semi-norm $\gamma$ on $V$ such that $\gamma(M) \leq 1$, it holds that $\exp(\gamma)(M_E) \leq 1$. We thus readily obtain the following.

\begin{proposition}\label{prop_topEqual}
Let $\Psi$ be a family of semi-norms on $V$. Then $\Psi$ is a fundamental family of semi-norms on $V$ if and only if $\exp(\Psi^*)$ is a fundamental family of semi-norms on $E$.
\end{proposition}

\begin{cor}\label{cor_permanence}
The space $E$ is Hausdorff (resp. metrizable, separable) if and only if $V$ is Hausdorff (resp. metrizable, separable).
\end{cor}

Whenever we speak of a topological space, we shall henceforth always assume it is Hausdorff.
The following result identifies $E$ with a subspace of $P(V) := \prod_{k \geq 0}V^{\widehat \otimes k}$. For $x \in P$, we write $x^k$ for the projection of $x$ onto $V^{\widehat\otimes k}$, so that $x=(x^0, x^1, x^2,\ldots)$.

\begin{cor}\label{cor_radiusOfConv}
Let $\Psi$ be a fundamental family of semi-norms on $V$. Then $E = \{x \in P \mid \forall \gamma \in \Psi^*, \sum_{k \geq 0} \gamma^{\otimes k}(x^k) < \infty \}$.
\end{cor}

By noting the identification $P^{\widehat \otimes 2} = \prod_{i,j \geq 0} V^{i, j}$, where $V^{i,j} \cong V^{\widehat \otimes (i+j)}$, the same considerations show that
\[
E^{\widehat\otimes 2} = \{x \in P^{\widehat \otimes 2} \mid \forall \gamma \in \Psi^*, \sum_{i,j \geq 0} \gamma^{\otimes (i+j)}(x^{i,j}) < \infty \}.
\]
Let $\rho^k : E \mapsto V^{\widehat\otimes k}$ denote the projection $\rho^k(x) = x^k$. The following result shall also be useful later and is another consequence of Proposition~\ref{prop_topEqual}.

\begin{cor}\label{cor_truncsConv}
The operators $T^{(n)} := \sum_{k=0}^n\rho^k : E \mapsto E$ converge uniformly on bounded sets to the identity operator on $E$.
\end{cor}

When $V$ is a normed space, we always equip $V^{\otimes k}$ with the projective norm unless stated otherwise. For an element $x \in P$ define its radius of convergence $R(x)$ as the radius of convergence of the series $\sum_{k \geq 0}\norm{x^k}\lambda^k$. Corollary~\ref{cor_radiusOfConv} then implies that $x \in E$ if and only if $R(x) = \infty$.

We now come to a more interesting permanence property. For a semi-normed space $(W, \gamma)$ denote the quotient normed space $W_\gamma = (W/\Ker(\gamma), \gamma)$ and $\widehat W_\gamma$ its completion. For a locally convex space $W$ and a Banach space $A$, a map $M \in \LLL(W, A)$ is called compact (resp. nuclear) if there exists a semi-norm $\gamma$ on $W$ such that the $\gamma(M) < \infty$ and the induced map $M_\gamma : \widehat W_\gamma \mapsto A$ is compact (resp. nuclear). Recall that $W$ is called Schwartz (resp. nuclear) if every $M \in \LLL(W,A)$ is compact (resp. nuclear) for every Banach space $A$.

\begin{proposition}\label{prop_SchwartzNuclear}
The space $E$ is Schwartz (resp. nuclear) if and only if $V$ is Schwartz (resp. nuclear).
\end{proposition}

We note that the case when $V$ is simply Schwartz shall not be used later in the paper and is recorded simply for completeness. Moreover nuclearity of $E$ shall only be applied in Section~\ref{subsec_convMeas} to the case $V = \R^d$. However the equivalent statement for $V = \R^d$ uses essentially the same proof and thus we record the result in full generality.

Let $\Psi$ be a fundamental family of sub-multiplicative semi-norms of a locally $m$-convex algebra $F$.
Equipping $\Psi^*$ with its natural partial order, $(\widehat F_\gamma)_{\gamma \in \Psi^*}$ is a projective system of Banach algebras and one obtains a dense topological algebra embedding $F \hookrightarrow \varprojlim_{\gamma \in \Psi^*} \widehat F_\gamma$ known as the Arens-Michael decomposition (see~\cite{Mallios86} Chapter III). As compact (resp. nuclear) operators form an operator ideal, we obtain the following.

\begin{lem}\label{lem_BanachAlg}
Let $F$ be a locally $m$-convex algebra. Then $F$ is Schwartz (resp. nuclear) if and only if every continuous algebra homomorphism $M : F \mapsto A$ is compact (resp. nuclear) for every Banach algebra $A$.
\end{lem}

For a normed space $V$ and Banach space $W$, denote by $\NN(V,W)$ the Banach space of nuclear operators from $V$ to $W$ with the nuclear norm $\normc_N$.

\begin{lem}\label{lem_nuclearExt}
Let $(V,\gamma)$ be a normed space and $A$ a Banach algebra. Let $M \in \NN(V, A)$ with $\norm{M}_N < 1$. Equip $T(V)$ with the norm $\exp(\gamma)$. Then the extension $M_E : T(V) \mapsto A$ is nuclear and $\norm{M_E}_N \leq (1-\norm{M}_N)^{-1}$.
\end{lem}

\begin{proof}
It holds that product map $M^{\otimes k} : V^{\otimes_\pi k} \mapsto A^{\otimes_\pi k}$ is nuclear with nuclear norm bounded by $\norm{M}_N^k$ (\cite{Holub70} Theorem~3.7 - the bound is clear from the proof therein), and the multiplication map $A^{\otimes_\pi k} \mapsto A$ has unit operator norm. It follows that $M^{\otimes k} : V^{\otimes_\pi k} \mapsto A$ is nuclear with nuclear norm at most $\norm{M}_N^k$ (\cite{Grothendieck55} p.84). The conclusion follows since $M_E = \sum_{k \geq 0} M^{\otimes k}$ is an absolutely convergent series in $\NN(T(V), A)$.
\end{proof}

For a semi-norm $\gamma$ on $V$, let $B_\gamma = \{v \in V \mid \gamma(v) < 1\}$, and for a subset $B \subseteq V$, let $\Gamma(B)$ be the absolutely convex hull of $B$.

\begin{proof}[Proof of Proposition~\ref{prop_SchwartzNuclear}]
The ``only if'' direction is clear. Let $A$ be a Banach algebra, $M \in \LLL(V,A)$, and let $\Psi$ be a fundamental family of semi-norms on $V$.
For a semi-norm $\gamma$ on $V$, recall that $B_{\gamma^{\otimes k}} = \Gamma(B_\gamma^{\otimes k}) \subset V^{\otimes k}$.

Suppose $V$ is Schwartz. Take $\gamma \in \Psi^*$ such that $M(B_\gamma) \subset A$ is relatively compact and $\gamma(M) < 1$. It follows that $M(B_{\gamma^{\otimes k}})$ is relatively compact in $A$ (\cite{Treves67} Proposition~7.11). Since the unit ball $B_{\exp(\gamma)}$ is given by $\Gamma(\bigcup_{k \geq 0} B_{\gamma^{\otimes k}})$, we obtain that $M(B_{\exp(\gamma)})$ is totally bounded in $A$. Thus $E$ is Schwartz by Lemma~\ref{lem_BanachAlg}.

Suppose $V$ is moreover nuclear. Take $\gamma \in \Psi^*$ such that the induced map $M_\gamma : V_\gamma \mapsto A$ is nuclear with $\norm{M_\gamma}_N < 1$. As $(V^{\otimes_\pi k})_{\gamma^{\otimes k}}$ and $(V_\gamma)^{\otimes_\pi k}$ are isometrically isomorphic (\cite{Grothendieck55} p.38), we have the natural identification $T(V)_{\exp(\gamma)}  \cong (T(V_\gamma), \exp(\gamma))$. It follows that $M_E : T(V)_{\exp(\gamma)} \mapsto A$ is nuclear by Lemma~\ref{lem_nuclearExt}. Thus $E$ is nuclear again by Lemma~\ref{lem_BanachAlg}.
\end{proof}

One may also ask when the extension map $\cdot_E : \LLL(V,A) \mapsto \Hom(E,A)$ is continuous under certain topologies. In the case of the strong topology when $V$ is normed, we obtain a homeomorphism by the following proposition. First, remark that if $\norm{x_j}\leq c$ and $\norm{x_j - y_j} \leq \varepsilon$ for $x_1,\ldots,x_n,y_1,\ldots,y_n \in A$, where $A$ is a normed algebra, then
\begin{equation}\label{eq_prodBound}
\norm{x_1\ldots x_n - y_1\ldots y_n}
\leq \sum_{j = 1}^n \binom{n}{j}\varepsilon^j c^{n-j}
= (c+\varepsilon)^n - c^n.
\end{equation}

\begin{proposition} \label{prop_contWithStrong}
Let $V$ be a normed space and $A$ a Banach algebra. The extension map $\cdot_E : M \mapsto M_E$ from $\LLL(V,A)$ to $\Hom(E,A)$ is continuous (and thus a homeomorphism) when one equips both sides with the strong topology.
\end{proposition}

\begin{proof}
Let $(M_j)_{j \geq 1} \rightarrow M$ in $\LLL(V,A)$. Let $\gamma$ be a norm on $V$ such that $\gamma(M)\leq 1$ and $\gamma(M_j) \leq 1$ for all $j \geq 1$.

Remark that for any bounded set $B \subset E$ and $\varepsilon > 0$, there exists $k_\varepsilon \geq 1$ such that $\sup_{x \in B} \gamma^{\otimes k}(x^k) \leq \varepsilon^k$ for all $k \geq k_\varepsilon$ (if not, then take a sequence $x_n \in B$ such that $\gamma^{\otimes n}(x^n_n) > \varepsilon^n$. Then $\exp(c\gamma)(x_n) > c^n\varepsilon^n$ for any $c > 1$ and $n \geq 1$, which is implies that $\exp(c\gamma)$ is not bounded on $B$ for some $c > 1$ which is a contradiction).

Remark that every bounded set in $V^{\otimes k}$ is contained in $\overline\Gamma(B_1\otimes\ldots \otimes B_k)$ for bounded sets $B_1,\ldots,B_k \subset V$, and that the supremum of a convex function on a set is equal to its supremum on the set's convex hull. Together with~\eqref{eq_prodBound}, this implies that for any fixed $n$,
\[
\sup_{x \in B} \sum_{0 \leq k \leq n} \norm{M^{\otimes k}_i(x^k) - M^{\otimes k}(x^k)} \rightarrow 0.
\]
Hence
\begin{equation*}
\begin{split}
\sup_{x \in B} \norm{M_j(x) - M(x)}
\leq & \sup_{x \in B} \sum_{0 \leq k \leq n} \norm{M^{\otimes k}_i(x^k) - M^{\otimes k}(x^k)} \\
&+ 2\sup_{x \in B} \sum_{k > n} \gamma^{\otimes k}{x^k}\
\end{split}
\end{equation*}
can be made arbitrarily small with sufficiently large $n$ and $j$.
\end{proof}

\begin{remark}
If we assume simply that $V$ is locally convex and $M \in \LLL(V,A)$, applying the above proposition to the semi-norm $\gamma(x) = \norm{M(x)}$ on $V$ implies in particular that the map $\lambda \mapsto (\lambda M)_E$ is continuous from $\C$ to $\Hom(E,A)$, where the latter is equipped with the strong topology.
\end{remark}

\section{Group-like elements}\label{sec_groupLike}

We recall that $T(V)$ is a Hopf algebra with coproduct $\Delta v = 1\otimes v + v\otimes 1$ for all $v \in V$
and antipode $\alpha(v_1\ldots v_k) = (-1)^k v_k \ldots v_1$ for all $v_1\ldots v_k \in V^{\otimes k}$ (\cite{FreeLieAlgebras} Proposition~1.10).

Consider now $V$ a locally convex space. Since $E^{\widehat \otimes 2}$ is itself a locally $m$-convex algebra (\cite{Mallios86} p.378), and since $\Delta \in \LLL(V, E^{\otimes_\pi 2})$, the extension $\Delta : E \mapsto E^{\widehat\otimes 2}$ is continuous by the universal property of $E$. Moreover the antipode $\alpha$ extends to a continuous linear map $\alpha : E \mapsto E$. This endows $E$ with an ``almost'' Hopf algebra structure (``almost'' since $E$ is not mapped to $E^{\otimes 2}$ under the coproduct $\Delta$ as for Hopf algebras, but to its completion $E^{\widehat\otimes 2}$).

Denote by $U(V) = \{g \in E \mid \alpha(g) = g^{-1}\}$ and  $G(V) = \{g \in E \mid \Delta(g) = g\otimes g, g \neq 0\}$ the groups of \emph{unitary} elements and \emph{group-like} elements of $E$ respectively. Note that since multiplication and inversion in $E$ are continuous (and indeed in every locally $m$-convex algebra,~\cite{Mallios86} p.5, p.52), $U$ and $G$ are topological groups when endowed with the subspace topology. Moreover, $U$ is closed in $E$ since the map $\phi: x \mapsto (\alpha(x)x, x\alpha(x))$ from $E$ into $E\times E$ is continuous and $U = \phi^{-1}\{(1,1)\}$. Likewise $G$ is closed in $E$ since $g^0 = 1$ for all $g \in G$ and $G = \psi^{-1}\{0\} \setminus\{0\}$ for the continuous map $\psi : x \mapsto x\otimes x - \Delta(x)$ from $E$ into $E^{\widehat \otimes 2}$. Finally, note the inclusion $G \subset U$.

In this section we collect several results concerning measures on $G$. While these results shall later be applied mostly to the case $V = \R^d$, we find making this assumption does not simplify the proofs, and thus make most statements in full generality.

All measures (resp. random variables) are assumed to be Borel. Denote by $\PP(S)$ the space of probability measures on a topological space $S$ endowed with the topology of weak convergence on $C_b(S, \C)$.

Recall that for a locally convex space $F$, an $F$-valued random variable $X$ is weakly (Gelfand-Pettis) integrable, or that $\EEE{X}$ exists, if $f(X)$ is integrable for all $f \in F'$ and if there exists $\EEE{X} := x \in F$ such that $\EEE{f(X)} = f(x)$. Letting $\mu$ be the probability measure associated with $X$, we denote by $\mu^* = \EEE{X}$ its barycenter. Unless stated otherwise, we shall always assume that $\mu$ is the measure associated to $X$ and that integrals are taken in the weak sense.

\begin{definition}\label{def_expSig}
For an $E$-valued random variable $X$, we call the sequence
\[
\ExpSig(X) := (\EEE{X^0}, \EEE{X^1}, \ldots) \in P = \prod_{k \geq 0} V^{\widehat\otimes k}
\]
the \emph{expected signature} of $X$ whenever $X^k$ is integrable for all $k \geq 0$.

When $V$ is normed, define $r_1(X)$ as the radius of convergence of the series
\[
\sum_{k \geq 0}\EEE{\norm{X^k}}\lambda^k
\]
(setting $r_1(X) = 0$ whenever $X^k$ is not norm-integrable for some $k \geq 0$), and $r_2(X)$ as the radius of convergence of the series
\[
\sum_{k \geq 0}\norm{\EEE{X^k}}\lambda^k,
\]
(setting $r_2(X) = 0$ whenever $X^k$ is not integrable for some $k \geq 0$).
\end{definition}

Note that $r_2(X) = R(\ExpSig(X))$. Remark also that $r_1(X) \leq r_2(X)$ and that Proposition~\ref{prop_radiiEquiv} provides a partial converse when $V = \R^d$ and $X$ is $G(\R^d)$-valued.

Note that $\ExpSig(X)$ exists whenever $X$ is integrable as an $E$-valued random variable. The following proposition now provides a converse when $X$ is $G$-valued. Recall that we identify $E$ as a subspace of $P$ (Corollary~\ref{cor_radiusOfConv}).

\begin{proposition}\label{prop_expSig}
Let $X$ be a $G$-valued random variable. Then $X$ is weakly integrable if and only if $\ExpSig(X)$ exists and lies in $E$. In this case $\EEE{X} = \ExpSig(X)$.
\end{proposition}

In the case that $V$ is normed, note that in order to conclude that a $G$-valued random variable $X$ is (weakly) integrable (as an $E$-valued random variable), Proposition~\ref{prop_expSig} implies that one only needs to check that each projection $X^k$ is (weakly) integrable and that $\norm{\EEE{X^k}}$ decays sufficiently fast as $k \rightarrow \infty$. Remark that this is certainly not true for an arbitrary $E$-valued random variable.

We observe that for any $f \in E'$, it holds that $f^{\otimes 2} \circ \Delta \in E'$ and $f(g)^2 = f^{\otimes 2} (\Delta g)$ for all $f \in E$ and $g \in G$. In particular, for all $\mu \in \PP(G)$, we have
\begin{equation}\label{eq_squareBound}
\mu(|f|) \leq \sqrt{\mu(f^2)} = \sqrt{\mu(f^{\otimes 2}\circ \Delta)}.
\end{equation}
This simple observation allows for very easy control of a measure through its barycenter. For example, whenever $\mu \in \PP(G)$ and $\EEE{X}$ exists, it follows immediately that for all $f \in E'$, the real random variable $f(X)$ has finite moments of all orders.

The main idea behind the proof of Proposition~\ref{prop_expSig} is that given the existence of $\EEE{X^k}$ for all $k \geq 0$, we wish to approximate $\EEE{f(X)}$ by $\sum_{k=0}^n \EEE{f(X^k)}$. Using the estimate~\eqref{eq_squareBound} and the grading of the coproduct $\Delta$, we apply dominated convergence to obtain $\EEE{f(X)} = \sum_{k\geq 0} \EEE{f(X^k)}$.

\begin{proof}[Proof of Proposition~\ref{prop_expSig}]
The ``only if'' direction is clear. Assume that $\ExpSig(X)$ exists and $\ExpSig(X) \in E$. As usual, let $\mu$ be the measure on $G$ associated to $X$. We are required to show that $f$ is $\mu$-integrable and that $\mu(f) = \gen{f, \ExpSig(X)}$ for all $f \in E'$.

We recall the projection $\rho^k : E \mapsto V^{\widehat\otimes k}$ and canonically embed $(V^{\widehat\otimes k})'$ into $E'$ by $f \mapsto f \rho^k =: f^k$ for all $f \in (V^{\widehat\otimes k})'$. By Corollary~\ref{cor_truncsConv}, $\sum_{k = 0}^n f^k$ converges uniformly on bounded sets (and a fortiori pointwise) to $f$.

Remark that for any $f \in E'$, $f \in (V^{\widehat\otimes k})'$ if and only if $f = f^k$. Recall that $\Delta$ is a graded linear map from $T(V)$ to $T(V)^{\otimes 2}$. In particular, for all $f_1 \in (V^{\widehat\otimes k})'$, $f_2 \in (V^{\widehat\otimes m})'$ and $x \in T(V)$, it holds that
\begin{equation}\label{eq_truncs}
(f_1\otimes f_2)\Delta(x) = (f_1\otimes f_2)\Delta (x^{k+m}).
\end{equation}
As $T(V)$ is dense in $E$,~\eqref{eq_truncs} holds for all $x \in E$, from which it follows that $(f_1 \otimes f_2)\circ \Delta \in (V^{\widehat\otimes (k+m)})'$.

Let $f \in E'$ and note that $\mu(f^k) = \gen{f^k, \EEE{X^k}}$ for all $k\geq 0$.
Since $\mu$ has support on $G$, it follows from~\eqref{eq_squareBound} and~\eqref{eq_truncs} that
\begin{equation}\label{eq_sumFinite}
\mu\left(\sum_{k\geq 0}|f^k|\right)
\leq \sum_{k\geq 0} \sqrt{\mu((f^k)^{\otimes 2} \circ \Delta)}
= \sum_{k\geq 0} \sqrt{(f^k)^{\otimes 2} \Delta \EEE{X^{2k}}}.
\end{equation}

Without loss of generality, we can assume that $|f(1)| \leq 1$. Let $\gamma$ be a semi-norm on $V$ such that $\exp(\gamma) \geq |f|$ and $\xi$ a semi-norm on $E$ such that $\xi \geq \exp(\gamma)^{\otimes 2}\circ\Delta$. It follows that $\exp(\gamma) \geq |f^k|$ for all $k \geq 0$, and thus $\xi \geq |(f^k)^{\otimes 2}\circ\Delta|$ for all $k \geq 0$.

Since $\ExpSig(X) \in E$, it follows from Corollary~\ref{cor_radiusOfConv} that $\sum_{k \geq 0} \sqrt{\xi(\EEE{X^k})}$ is finite, and hence~\eqref{eq_sumFinite} is finite. By dominated convergence, we obtain
\[
\mu(f) = \lim_{n \rightarrow\infty} \mu(\sum_{k=0}^n f^k).
\]
It then follows that $\mu(f) = \gen{f, \ExpSig(X)}$ as desired since
\[
\mu(\sum_{k=0}^n f^k) = \sum_{k=0}^n \gen{f^k, \EEE{X^k}} \rightarrow \gen{f,\ExpSig(X)}.
\]
\end{proof}

\begin{cor}\label{cor_expExists}
Let $V$ be a normed space and $X$ a $G$-valued random variable. Then $\EEE{X} \in E$ exists if and only if $r_2(X) = \infty$, i.e., $\ExpSig(X)$ exists and has an infinite radius of convergence. In this case $\EEE{X} = \ExpSig(X)$.
\end{cor}

We are moreover able to show explicit bounds between $r_1(X)$ and $r_2(X)$ when $V = \R^d$. Suppose first that $V$ is a normed space. Remark that $\norm{\Delta v} = 2\norm{v}$ for all $v \in V$, from which it follows that $\norm{\Delta\mid_{V^{\widehat\otimes k}}} = 2^k$ and thus
\begin{equation}\label{eq_coprodIneq}
\norm{\Delta x^k} \leq 2^k\norm{x^k} \text{ for all } x \in E.
\end{equation}
Let $V = \R^d$ equipped with the $\ell^1$ norm from its standard basis $e_1,\ldots, e_d$, and denote $e_I = e_{i(1)}\ldots e_{i(k)} \in V^{\otimes k}$ for a word $I=i(1)\ldots i(k)$ in the alphabet $\{1,\ldots, d\}$. Then the grading of $\Delta$ gives
\begin{equation*}
\begin{split}
\EEE{\norm{X^k}^2}
= \EEE{\left(\sum_{|I|=k} |\gen{e_I, X^k}|\right)^2}
&\leq d^k \EEE{\sum_{|I| = k}\gen{e_I, X^k}^2} \\
& = d^k \sum_{|I| = k} e_I^{\otimes 2} \Delta \EEE{X^{2k}} \\
& \leq d^k \norm{\Delta\EEE{X^{2k}}},
\end{split}
\end{equation*}
where the last inequality follows since $(e_I \otimes e_J)_{|I|=|J|=k}$ is an $\ell^1$ basis for $V^{\otimes 2k}$. Using~\eqref{eq_coprodIneq} we now obtain the following.

\begin{proposition}\label{prop_radiiEquiv}
Let $X$ be a $G(\R^d)$-valued random variable. It follows that $\EEE{\norm{X^k}^2} \leq d^k 2^{2k} \norm{\EEE{X^{2k}}}$. In particular, $r_1(X) \leq r_2(X) \leq 2\sqrt{d} r_1(X)$.
\end{proposition}

\section{Representations}\label{sec_reps}

Recall that for any Hopf algebra, one may define the tensor product and dual of representations via the coproduct and antipode by $M_1\otimes M_2 (x) := (M_1\otimes M_2) \Delta(x)$ and $M^*(x) := M(\alpha(x))^*$. By virtue of continuity of $\Delta$ and $\alpha$, we observe that the family of continuous representations of $E$ over finite dimensional Hilbert spaces is closed under tensor products and duals.

\begin{definition}
Denote by $\A(V)$ the family of finite dimensional representations of $E$ which arise from extensions of all linear maps $M \in \LLL(V,\uu(H_M))$, where $H_M$ ranges over all finite dimensional Hilbert spaces and $\uu(H_M)$ denotes the Lie algebra of the anti-Hermitian operators on $H_M$. Denote by $\CC(V)$ the set of corresponding matrix coefficients, i.e., the set of linear functionals $M_{u,v} \in \LLL(E,\C)$, $M_{u,v}(x) = \gen{M(x)u,v}$ for all $M \in \A$ and $u,v \in H_M$
\end{definition}

The family $\A$ possesses the desirable property that it is closed under taking tensor products and duals of representations. Moreover, we see that $\A$ contains exactly those finite dimensional representations of $E$ which preserve involution, i.e., $M(\alpha x) = M(x)^*$ for all $x \in E$. It follows that every $M \in \A$ is a unitary representation of the group $U$, and thus of $G$.

Observe that the tensor product $M_1 \otimes M_2$ (of any representations $M_1, M_2$ of $E$) coincides on $G$ with the usual group-theoretic tensor product of representations. Moreover, the dual representation $M^*$ of $M \in \A$ can be identified on $U$ with the conjugate representation of $M$ on $U$. It follows that $\CC\mid_G$ forms a $*$-subalgebra of $C_b(G, \C)$.

Let $S$ be a topological space and $F$ a separating $*$-subalgebra of $C_b(S,\C)$. Recall that for tight Borel measures $\mu$ and $\nu$ on $S$, it follows from the Stone-Weierstrass theorem that $\mu = \nu$ if and only if $\mu(f) = \nu(f)$ for all $f \in F$ (\cite{Bogachev07} Exercise~7.14.79). We now obtain the following from the above discussion.

\begin{lem}\label{lem_compactOpen}
Assume that $\A$ separates the points of $G$. Then for tight Borel measures $\mu, \nu$ on $G$, $\mu = \nu$ if and only if $\mu(f) = \nu(f)$ for all $f \in \CC$, or equivalently, $\mu(M) = \nu(M)$ for all $M \in \A$.
\end{lem}

We show in Theorem~\ref{thm_sepPointsUnitary} that in fact $\A(\R^d)$ separates the points of $E(\R^d)$.

\subsection{Separation of points}

We investigate conditions under which algebra homomorphisms of $E$ separate points. Though ultimately we apply the theory to the case $V = \R^d$, the arguments used in the general case are exactly the same and we provide them here.

For a Banach algebra $A$ and $M \in \LLL(V,A)$, let $(\lambda M)$ denote the algebra homomorphism on $E$ induced by $\lambda M \in \LLL(V,A)$ ($\lambda$ possibly complex if $A$ is over $\C$). For $\lambda \in \R$, let $\delta_\lambda: E\mapsto E$ denote the dilation operator $\delta_\lambda(x^0,x^1,\ldots) = (\lambda^0x^0, \lambda^1x^1,\ldots)$ (note that $(\lambda M) = M\delta_\lambda$ for $\lambda \in \R$).

\begin{lem}\label{lem_sepSubspace}
Let $V$ be locally convex, $A$ a Banach algebra and $M \in \LLL(V,A)$. Let $x \in E$ such that $M(x^k) \neq 0$ for some $k \geq 0$. Then there exists $\varepsilon > 0$ sufficiently small such that $(\varepsilon M)(x) \neq 0$.
\end{lem}

\begin{proof}
Since $\norm{M(x)}$ is a semi-norm on $E$, $\sum_{k \geq 0} \norm{M(x^k)}$ converges by Corollary~\ref{cor_radiusOfConv}, from which the conclusion follows.
\end{proof}

Let $\F$ be a field and $A$ an $\F$-algebra. A polynomial identity over $\F$ on a subset $Q \subseteq A$ is a polynomial in non-commuting indeterminates $x_1,\ldots, x_k$, with coefficients in $\F$, which is non-zero (that is, not every coefficient is zero) and which vanishes under all substitutions of variables $x_1,\ldots, x_k \in Q$. We refer to Giambruno and Zaicev~\cite{polyIden} for further details.

Let $V$ be a vector space with Hamel basis $\Theta$. Then the set of pure tensors $\Theta^{\otimes k} = \{v_1 \ldots v_k \mid v_j \in \Theta, 1\leq j\leq k\}$ is a Hamel basis for $V^{\otimes k}$.
Thus for every $x \in V^{\otimes k}$ define $\Theta_x$ as the finite set of vectors in $\Theta$ which appear in the representation of $x$ in the basis $\Theta^{\otimes k}$. Define $f^\Theta_x$ the canonical formal non-commuting polynomial in indeterminates $\Theta_x$ associated with $x$. As $\Theta_x$ is a finite set, the following is a consequence of the Hahn-Banach theorem.

\begin{lem}\label{lem_nonZero}
Let $V$ be a locally convex space with Hamel basis $\Theta$, $A$ an algebra which is a topological vector space, and $Q \subseteq A$ a subset. Let $k \geq 0$ and $x \in V^{\otimes k}$. The following two assertions are equivalent.
\begin{enumerate}[label=\upshape(\roman*\upshape)]
\item \label{point_1} $f^\Theta_x$ is not a polynomial identity over $\R$ on $Q$.
\item \label{point_2} There exists a continuous linear map $M : V \mapsto \spn Q$ such that $M(x)$ is non-zero and $M(v)$ is in $Q$ for all $v \in \Theta_x$.
\end{enumerate}
\end{lem}

\begin{remark}
If one is not interested in the topological aspects, the same statement holds if one replaces $\R$ by a field $\F$, $V$ by a vector space over $\F$, $A$ by an $\F$-algebra, and drops the continuity assumption in~\ref{point_2}.
\end{remark}

\subsection{Polynomial identities over Lie algebras}

From Lemmas~\ref{lem_sepSubspace} and~\ref{lem_nonZero}, it is clear that to study how representations in $\A(\R^d)$ separate the points of $E(\R^d)$, we must look at polynomial identities in unitary Lie algebras.
Let $m \geq 1$ be an integer and denote by $\cdot^s$ the symplectic involution on $M_{2m}(\C)$, which we recall is an involution of the first kind (see~\cite{Giambruno95}).

Recall the real Lie algebra $\symp(m) = \{u \in \uu(\C^{2m}) \mid u^s + u = 0\}$ ($\symp(m)$ is the Lie algebra of the compact symplectic group $Sp(m)$).
A closely related complex Lie subalgebra of $\gl(\C^{2m})$ is $\symp(m,\C) = \{u \in M_{2m}(\C) \mid u^s + u= 0\}$.
It holds that $\symp(m,\C)$ is the complexification of $\symp(m)$.

We now illustrate our interest in the Lie algebras $\symp(m)$ and $\symp(m,\C)$. From the remark that $\symp(m,\C) = \{u - u^s \mid u \in M_{2m}(\C)\}$, we may reformulate a result due to Giambruno and Valenti as follows.

\begin{theorem}[\cite{Giambruno95} Theorem~6] \label{thm_3n}
Let $m \geq 2$ and $f(x_1,\ldots,x_k)$ a polynomial identity over $\C$ on $\symp(m,\C) \subset M_{2m}(\C)$. Then $\deg(f) > 3m$.
\end{theorem}

The following is a slight generalization of~\cite{polyIden} Theorem~1.3.2 and follows from exactly the same inductive proof.

\begin{lem}\label{lem_idenOnSubspace}
Let $\F$ be an infinite field, $A$ an $\F$-algebra and $Q$ a linear subspace of $A$. If $f$ is a polynomial identity over $\F$ on $Q$, then every multi-homogeneous component of $f$ is a polynomial identity over $\F$ on $Q$.
\end{lem}

We remark that every multi-homogeneous polynomial identity over $\C$ (and a fortiori over $\R$) on $\symp(m)\subset M_{2m}(\C)$ is also a polynomial identity over $\C$ on its complexification $\symp(m,\C)$. Thus if $f$ is a polynomial identity over $\R$ on $\symp(m)$ for $m \geq 2$, then by Theorem~\ref{thm_3n} and Lemma~\ref{lem_idenOnSubspace}, every multi-homogeneous component of $f$ has degree greater than $3m$. Together with Lemmas~\ref{lem_sepSubspace} and~\ref{lem_nonZero}, we have the following result.

\begin{theorem}\label{thm_sepPointsUnitary}
Let $x \in E(\R^d)$ such that $x^k \neq 0$ for some $k \geq 0$. Then for any integer $m \geq \max\{2,k/3\}$ there exists $M \in \LLL(\R^d, \symp(m))$ such that $M(x) \neq 0$.
In particular, $\A(\R^d)$ separates the points of $E(\R^d)$.
\end{theorem}

\begin{remark}
The necessity that $V = \R^d$ only came into the above argument to ensure that $V^{\otimes k} = V^{\widehat \otimes k}$. If one was able to find an analogue of
Lemma~\ref{lem_nonZero} for elements $x \in V^{\widehat \otimes k}$, or an analogue of
Theorem~\ref{thm_3n} for appropriate series of polynomials of bounded degree but an unbounded number of indeterminates, then one could readily extend Theorem~\ref{thm_sepPointsUnitary} to the case when $V$ is infinite dimensional.
\end{remark}

\begin{cor}\label{cor_MAP}
The group $U(\R^d)$ is maximally almost periodic.
\end{cor}

\begin{remark}\label{remark_notLocCompact}
For $d \geq 2$, the topological group $G(\R^d)$ (and thus $U(\R^d)$) is not locally compact. To observe this, let $V = \R^d$ and $L(V)$ be the smallest Lie algebra in $T(V)$ containing $V$.
Since every $\ell \in L(V)$ satisfies $\Delta(\ell) = 1\otimes \ell + \ell\otimes 1$ (\cite{FreeLieAlgebras} Theorem~1.4), a direct calculation shows that $\exp(\ell) \in G$.

Let $u,v \in V$ be linearly independent elements and $W = \spn{u,v}$. Observe that $L(W)$ contains a non-zero element in $W^{\otimes k}$ for every $k \geq 1$. In light of Proposition~\ref{prop_topEqual}, for any neighborhood of zero $B$ of $L(W)$ one can construct a sequence $(\ell_n)_{n\geq 1} \in B$ such that $\gamma(\exp(\ell_i) - \exp(\ell_j)) \geq 1$ for all $i\neq j$ and some semi-norm $\gamma$ on $E$. Since $\exp: L(W) \mapsto G$ is continuous (\cite{Arens46} Theorem~3), it follows that no neighborhood of the identity in $G$ is contained in a sequentially compact set (the same argument more generally applies whenever $V$ is metrizable).
\end{remark}

It follows from Corollary~\ref{cor_permanence} that $E$ is Polish whenever $V$ is metrizable and separable, and thus $G$, as a closed subset of $E$, is also Polish. By Lemma~\ref{lem_compactOpen} and Theorem~\ref{thm_sepPointsUnitary} we have the following.

\begin{cor}\label{cor_uniqueMeasure}
For Borel probability measures $\mu$ and $\nu$ on $G(\R^d)$, it holds that $\mu = \nu$ if and only if $\mu(f) = \nu(f)$ for all $f \in \CC(\R^d)$, or equivalently, $\mu(M) = \nu(M)$ for all $M \in \A(\R^d)$.
\end{cor}

For a Borel probability measure $\mu$ on $G(\R^d)$, with associated random variable $X$, we are thus able to define its characteristic function (or Fourier transform) by $\phi_X = \widehat \mu := \mu|_{\A}$, which uniquely characterizes $\mu$.

\section{Signatures of paths}\label{sec_signature}

We now discuss the space $E$ in the setting of rough paths theory. The main connection is that the signature of any geometric rough path on $\R^d$ lies in $G(\R^d)$.
We treat rough paths in the sense of Lyons and refer to~\cite{FrizVictoir10} and~\cite{LyonsQian02} for details and terminology.

Let $V$ be a Banach space, $p\geq 1$, $T > 0$, and $\Delta_{[0,T]} = \{(s,t)\mid 0\leq s\leq t\leq T\}$. Let $\omega$ denote a control function and $T^n(V) = \bigoplus_{0 \leq k \leq n} V^{\widehat\otimes k}$ the truncated tensor algebra. We recall that the space of $p$-rough paths $\Omega_p(V)$ is the collection of all continuous multiplicative maps $\x : \Delta_{[0,T]} \mapsto T^{\floor{p}}$ with $p$-variation controlled by some control $\omega$, that is,
\begin{enumerate}[label=\upshape(\alph*\upshape)]
\item \label{point_mult} $\x^0_{s,t} = 1$ and $\x_{s,t}\x_{t,u} = \x_{s,u}$ for all $0\leq s\leq t \leq u \leq T$, and
\item \label{point_pVarCont} for some control $\omega$ one has
\begin{equation}\label{eq_factorialBound}
\sup_{0 \leq k \leq \floor{p}} \left((k/p)!\beta_p\norm{\x^k_{s,t}}\right)^{p/k} \leq \omega(s,t), \; \forall (s,t) \in \Delta_{[0,T]},
\end{equation}
\end{enumerate}
where $\beta_p$ is a constant that only depends on $p$.

The map $\x$ may alternatively be viewed as a path $\x_{0,\cdot} : [0,T] \mapsto T^{\floor{p}}, t \mapsto \x_{0,t}$ of finite $p$-variation, that is,
\begin{equation}\label{eq_pVarDef}
\norm{\x}_{\pvar;[0,T]}
:= \sum_{0 \leq k \leq \floor{p}} \sup_{\DD \subset [0,T]} \left( \sum_{t(j) \in \DD} \left((k/p)!\beta_p\norm{\x^k_{t(j),t(j+1)}}\right)^{p/k} \right)^{1/p}
\end{equation}
is finite, which completely characterizes $\x$ due to the multiplicative property~\ref{point_mult} (noting that $\x_{s,t} = \x_{0,t}\x_{0,s}^{-1}$).

Let $\x \in \Omega_p$ satisfy~\eqref{eq_factorialBound} for some control $\omega$. A fundamental result of rough paths theory is that for all $n \geq \floor{p}$ there exists a unique lift $\Lyons_n(\x): \Delta_{[0,T]}\mapsto T^{n}$ such that~\ref{point_mult} and~\ref{point_pVarCont} remain true for the same $\omega$ and with $\sup_{0 \leq k \leq \floor{p}}$ replaced by $\sup_{0 \leq k \leq n}$ in~\eqref{eq_factorialBound} (\cite{LyonsQian02} Theorem~3.1.2). Equivalently, there exists a unique lift to the entire product space $\Lyons(\x) : \Delta_{[0,T]}\mapsto P = \prod_{k \geq 0} V^{\widehat \otimes k}$ such that~\ref{point_mult} and~\ref{point_pVarCont} remain true for the same $\omega$ and with $\sup_{0 \leq k \leq \floor{p}}$ replaced by $\sup_{0 \leq k}$ in~\eqref{eq_factorialBound}.

An immediate consequence of the factorial decay in~\eqref{eq_factorialBound} is that the lift $\Lyons(\x)$ takes values in the space $E$ for any $p \geq 1$ (see Corollary~\ref{cor_radiusOfConv}).

\begin{remark}
While the value of $\beta_p$ does not affect the definition of the space $\Omega_p$, its existence is crucial to ensure the factorial decay arising from the lift. On this point, we mention the work of Hara and Hino~\cite{HaraHino10} who have resolved a conjecture on the optimal possible value of $\beta_p$.
\end{remark}

We thus make a canonical extension of the space $\Omega_p$.

\begin{definition}
Define the space $\Omega E_p$ as the set of maps $\x : \Delta_{[0,T]} \mapsto E$ which satisfy~\ref{point_mult} and~\ref{point_pVarCont} with $\sup_{0 \leq k \leq \floor{p}}$ replaced by $\sup_{0 \leq k}$ in~\eqref{eq_factorialBound}.
\end{definition}

It follows that the lift $\Lyons$ is a bijective map from $\Omega_p$ to $\Omega E_p$, with inverse provided naturally by the $\floor{p}$-th level truncation $(\x^0_{s,t}, \x^1_{s,t}, \ldots) \mapsto (\x^0_{s,t}, \x^1_{s,t}, \ldots, \x^{\floor{p}}_{s,t})$.

The element $\Lyons(\x)_{0,T} \in E$ is called the \emph{signature} of a rough path $\x \in \Omega_p$. For $1 \leq p < 2$, $\Lyons(\x)_{0,T}$ is precisely the sequence of iterated integrals of the path $\x_{0,\cdot} : [0,T] \mapsto V$ taken in the sense of Young.

\begin{remark}
The only property of the projective tensor norm used above is that the projective extension provides a sub-multiplicative system of norms. Completely analogous results hold true if one equips $T(V)$ with any system of sub-multiplicative norms and defines $E$ as the completion of $T(V)$ under scalar dilations of these norms. Note that in the case $V = \R^d$, all these systems lead to identical definitions and topologies on the space $E$.
\end{remark}

The lift $\Lyons$ moreover exhibits a natural continuity property with respect to the $p$-variation topology on $\Omega_p$. For $\x, (\x(n))_{n \geq 1} \in \Omega_p$, a control $\omega$ and a sequence of positive reals $(a_n)_{n \geq 1}$ with $a_n \geq 1$, consider the statement
\begin{equation}\label{eq_pTopDef}
\begin{split}
&\omega \text{ controls the } p\text{-variation of } \x \text{ and } \x(n) \text{ for all $n \geq 1$, and } \\
&\sup_{0 \leq k \leq \floor{p}} \left((k/p)!\beta_pa_n\norm{\x(n)^k_{s,t} - \x^k_{s,t}}\right)^{p/k} \leq \omega(s,t), \; \; \forall (s,t)\in \Delta_{[0,T]}.
\end{split}
\end{equation}

When~\eqref{eq_pTopDef} is satisfied for some control $\omega$ and a sequence $(a_n)_{n \geq 1}$ such that $a_n \geq 1$ and $a_n \rightarrow \infty$, we say that $\x(n) \rightarrow \x$ in the $p$-variation topology of $\Omega_p$. One makes the same definition for $\x, (\x(n))_{n \geq 1}\in \Omega E_p$ with $\sup_{0 \leq k \leq \floor{p}}$ replaced by $\sup_{0 \leq k}$ in~\eqref{eq_pTopDef}.

The following is an immediate consequence of the continuity of the individual lifts $\Lyons_n$ for $n \geq \floor{p}$ (\cite{LyonsQian02} Theorem~3.1.3).

\begin{proposition}
If $\x, (\x(n))_{n \geq 1} \in \Omega_p$ satisfy~\eqref{eq_pTopDef} for some $\omega$ and $(a_n)_{n \geq 1}$ with $a_n \geq 1$, then $\Lyons(\x), (\Lyons(\x(n)))_{n\geq 1} \in \Omega E_p$ satisfy~\eqref{eq_pTopDef} for the same control $\omega$ and sequence $(a_n)_{n \geq 1}$.

In particular, $\Lyons$ is continuous (and thus a homeomorphism) when $\Omega_p$ and $\Omega E_p$ are equipped with their respective $p$-variation topologies.
\end{proposition}

We equip $\Omega_p$ with the $p$-variation topology and denote the evaluation map $\II^p_{[0,T]}: \Omega_p \mapsto E$, $\x \mapsto \Lyons(\x)_{0,T}$.

\begin{cor} \label{cor_contD1}
The map $\II^p_{[0,T]}$ is continuous.
\end{cor}

Recall that the space of geometric $p$-rough paths $G\Omega_p$ is the closure of $S_{\floor p}(\Omega_1)$ in $\Omega_p$. 
\begin{definition}
For $p \geq 1$, define $S_p(V) = \{\Lyons(\x)_{0,T} \mid \x \in G\Omega_p\} \subset E$ as the set of signatures of all geometric $p$-rough paths.
\end{definition}
We equip $S_p$ with the subspace topology from $E$. Observe that $S_1$ is dense in $S_p$ as a consequence of Corollary~\ref{cor_contD1}.

Remark that $S_p$ is closed under multiplication in $E$ and that for all $\x \in G\Omega_p$, the inverse of $\Lyons(\x)_{0,T}$ is $\Lyons(\y)_{0,T} = \alpha (\Lyons(\x)_{0,T})$, where $\y \in G\Omega_p$ is the reversal of $\x$ and $\alpha$ is the antipode of $E$ defined in Section~\ref{sec_groupLike} (\cite{LyonsQian02} Theorem~3.3.3). Thus $S_p$ is a subgroup of $U = \{g \in E \mid \alpha(g) = g^{-1}\}$.

\subsection{Finite dimensional case}\label{subsec_finiteDimCase}

In this section we consider $V = \R^d$. It follows that $P(\R^d)$ (resp. $E(\R^d)$) can be identified with the algebra of non-commuting formal power series in $d$ indeterminates (resp. with an infinite radius of convergence).

We remark that the coproduct $\Delta$ of $E(\R^d)$ is given by a locally finite formula involving the shuffle product (\cite{FreeLieAlgebras} Proposition~1.8) and an element $g \in E(\R^d)$ is in $G(\R^d)$ precisely when $(g^0, g^1, \ldots, g^n)$ is in the free $n$-step nilpotent Lie group $G^{n}(\R^d)$ for all $n \geq 1$ (\cite{Lyons07} Lemma~2.24).

A fundamental result of Chen~\cite{Chen57} is that the signature of a bounded variation path in $\R^d$ is a group-like element of $E(\R^d)$ (see also~\cite{Lyons07} Section~2.2.5), and thus $S_1(\R^d) \subset G(\R^d)$.
Since $G$ is closed in $E$, we immediately obtain the inclusions $S_p(\R^d) \subset \overline{S_1(\R^d)} \subseteq G(\R^d)$ for all $p \geq 1$.

A closely related set to $G\Omega_p(\R^d)$ is the space $WG\Omega_p(\R^d) \subset \Omega_p(\R^d)$ of \emph{weakly} geometric $p$-rough paths, that is, those $p$-rough paths $\x\in \Omega_p(\R^d)$ which take values in the free $\floor{p}$-step nilpotent Lie group, i.e.,
\[
(\x^0_{s,t},\x^1_{s,t}, \ldots, \x^{\floor{p}}_{s,t}) \in G^{\floor{p}}(\R^d), \; \forall(s,t) \in \Delta_{[0,T]}.
\]
We note the strict inclusions $G\Omega_p(\R^d) \subset WG\Omega_p(\R^d) \subset G\Omega_{p'}(\R^d)$ for any $p' > p$ (\cite{FrizVictoir10} Section~8.5), and thus $WS_p(\R^d) \subset S_{p'}(\R^d)$, where $WS_p(\R^d) = \{\Lyons(\x)_{0,T} \mid \x \in WG\Omega_p(\R^d)\}$. Thus all results stated for the sets $S_p(\R^d)$ have analogous versions for the sets $WS_p(\R^d)$.

\begin{proposition}\label{prop_sigmaCompact}
Let $p \geq 1$. Then $S_p(\R^d)$ is $\sigma$-compact in $G(\R^d)$.
In particular, $S_p(\R^d)$ is a Borel set of $G(\R^d)$.
\end{proposition}

For the proof, we recall the (homogeneous) $p$-variation metric $d_{\pvar}$ on $\Omega_p(\R^d)$ under which $(\Omega_p(\R^d), d_{\pvar})$ is a complete metric space with a coarser topology than the $p$-variation topology, but for which convergence of a sequence in $d_{\pvar}$ implies the existence of a subsequence which converges in the $p$-variation topology (see \cite{FrizVictoir10} Section~8, \cite{LyonsQian02} Proposition~3.3.3, but note the differing notations for homogeneous and inhomogeneous metrics in the two texts; we use the notation of~\cite{FrizVictoir10}).

\begin{proof}
For $r > 0$, consider the set $B_p^r = \{\x \in G\Omega_p(\R^d) \mid \norm{\x}_{\pvar;[0,T]} \leq r\}$. For every $\x \in B_p^r$ there exists a suitable reparametrization $\y \in B_p^r$ for which $\norm{\x}_{\pvar;[0,T]} = \norm{\y}_{\pvar;[0,T]}$, $\Lyons(\x)_{0,T} = \Lyons(\y)_{0,T}$, and $t \mapsto \y_{0,t}$ is $(1/p)$-H{\"o}lder continuous with H{\"o}lder coefficient depending only on $r$ and $p$. Let $C_p^r \subset B_p^r$ be the set of all such reparametrizations.

Let $p' > p$ be such that $\floor {p'} = \floor p$. It follows from an interpolation estimate and the Arzel{\`a}-Ascoli theorem (\cite{FrizVictoir10} Lemma~5.12, Proposition~8.17) that $C_p^r$ is compact in $(G\Omega_{p'}(\R^d), d_{\pprimevar})$ and thus sequentially compact in $G\Omega_{p'}(\R^d)$ under the $p'$-variation topology.

Since $\II^{p'}_{[0,T]} : G\Omega_{p'}(\R^d) \mapsto S_{p'}(\R^d)$ is continuous by Corollary~\ref{cor_contD1}, and $\II^{p'}_{[0,T]}(C_p^r) = \II^{p'}_{[0,T]}(B_p^r)$, it follows that $\II^{p'}_{[0,T]}(B_p^r)$ is sequentially compact in $S_{p'}(\R^d)$, and thus compact. Since $S_p(\R^d) = \bigcup_{r\geq 1} \II^{p'}_{[0,T]} (B_p^r)$, it follows that $S_p(\R^d)$ is $\sigma$-compact in $G(\R^d)$.
\end{proof}

We lastly record here a consequence of Theorem~\ref{thm_sepPointsUnitary} and Theorem~4 of~\cite{Hambly10}, which strengthens Corollary~1.7 therein, and which was originally observed by Prof. Thierry L{\'e}vy.

\begin{cor}\label{cor_compactSep}
A path of bounded variation in $\R^d$ is tree-like if and only if its Cartan development into every finite-dimensional compact Lie group is trivial.
\end{cor}

\section{Expected signature}\label{sec_expectedSign}

Our main focus in this section is the expected signature of $G$-valued random variables and its connection with the characteristic function defined at the end of Section~\ref{sec_reps}.

\subsection{Moments problem}\label{subsec_momentsProb}

In this section we study the moments problem for $G(\R^d)$-valued random variables, that is, conditions under which a $G(\R^d)$-valued random variable is uniquely determined by its expected signature. We mention here that a large part of the results in this section arose from discussions with Dr. Ni Hao, and we hope to soon jointly expand on this material in a future paper.

When $V$ is a normed space, recall from Corollary~\ref{cor_expExists} that if $X$ is a $G$-valued random variable such that $\ExpSig(X)$ exists and has an infinite radius of convergence, then $\EEE{X}$ exists as an element of $E$ and is equal to $\ExpSig(X)$. Thus $\EEE{f(X)}$ is completely determined by $\ExpSig(X)$ for all $f \in E'$, and in particular for all $M \in \A$. The following is now a consequence of the uniqueness of probability measures from Corollary~\ref{cor_uniqueMeasure}.

\begin{proposition}\label{prop_infRadUnique}
Let $X$ and $Y$ be $G(\R^d)$-valued random variables such that $\ExpSig(X) = \ExpSig(Y)$ and $\ExpSig(X) \in E$, i.e., $\ExpSig(X)$ has an infinite radius of convergence. Then $X \eqd Y$.
\end{proposition}

Recall from Corollary~\ref{cor_contD1} that the evaluation map $\II^p_{[0,T]}: \Omega_p \mapsto E$ is continuous. It follows that the signature $\Lyons(\X)_{0,T}$ of any $\Omega_p$-valued (resp. $G\Omega_p$-valued) random variable $\X$ is a well-defined (Borel) $E$-valued (resp. $U$-valued, or $G(\R^d)$-valued in case $V = \R^d$) random variable. 

\begin{example}\label{ex_LevyUnique}
We apply Proposition~\ref{prop_infRadUnique} to the L{\'e}vy–Khintchine formula established in~\cite{FrizShekhar12}.
Recall that every L{\'e}vy process in $\R^d$ admits a natural lift to a $G\Omega_p(\R^d)$-valued random variable $\X$ for any $p > 2$ by adding appropriate adjustments for jumps (see~\cite{Williams01} Section~2). Let $(a,b,K)$ denote the triplet of the L{\'e}vy process.

It follows from~\cite{FrizShekhar12} Section~9.1 that $\ExpSig(X)$ exists (as an element of $P(\R^d) = \prod_{k \geq 0} (\R^d)^{\otimes k}$) whenever the L{\'e}vy measure $K$ has finite moments of all orders. Furthermore, $\ExpSig(X) \in E$ exactly when
\begin{equation}\label{eq_expIntegral}
\int_{\R^d}\left(e^{\lambda \norm{y}} - 1 - \lambda\1_{\norm{y}\leq 1}\norm{y} \right) K(dy) < \infty \; \textnormal{ for all $\lambda > 0$}.
\end{equation}
It follows by Proposition~\ref{prop_infRadUnique} that whenever~\eqref{eq_expIntegral} is satisfied, $\Lyons(\X)_{0,T}$ is uniquely determined as a $G(\R^d)$-valued random variable by its expected signature.
\end{example}

Recall the radius of convergence $r_1(X)$ from Definition~\ref{def_expSig}. Theorem~\ref{thm_r1Pos} below provides sufficient conditions to ensure that $r_1(X) > 0$ or $r_1(X) = \infty$ without explicit knowledge of $\ExpSig(X)$.

For a subset $B \subseteq A$ of an algebra $A$ and $n \geq 1$, define $B^n = \{x_1\ldots x_n \mid x_1,\ldots, x_n \in B\}$. For an element $x \in A$, define $B(x) = \inf\{n \geq 1 \mid x \in B^{n}\}$ (taking $B(x) = \infty$ if $x\notin B^n$ for all $n \geq 1$).

Note that for a topological algebra $A$ with (jointly) continuous multiplication, an $A$-valued random variable $X$, and a (Borel) measurable set $B \subset A$, $B(X)$ is a well-defined random variable in $\{1,2,\ldots\}\cup\{\infty\}$.

\begin{theorem}\label{thm_r1Pos}
Let $V$ be a normed space and $X$ an $E$-valued random variable. Suppose there exists a bounded, measurable set $B \subset E$ such that $B(X)$ has an exponential tail, i.e., $\EEE{e^{\lambda B(X)}} < \infty$ for some $\lambda >0$. Then $r_1(X) > 0$. If moreover $\EEE{e^{\lambda B(X)}} < \infty$ for all $\lambda > 0$, then $r_1(X) = \infty$.
\end{theorem}

\begin{proof}
Equip $E$ with the projective extension of the norm on $V$. For any $r > 0$ and $\lambda > 0$ such that $\sup_{x \in B}\norm{\delta_r (x)} < e^\lambda$, it holds that
\begin{equation}\label{eq_r1Bound}
\sum_{k\geq 0} r^k \EEE{\norm{X^k}} = \EEE{\norm{\delta_r(X)}} \leq \EEE{e^{\lambda B(X)}},
\end{equation}
where the inequality follows from the fact that $\delta_r(X) = \delta_r(X_1)\ldots \delta_r(X_{B(X)})$ for some $X_1,\ldots, X_{B(X)} \in B$.

Suppose first that $\EEE{e^{\lambda B(X)}} < \infty$ for all $\lambda > 0$. For any $r > 0$ let $\lambda > 0$ be sufficiently large such that $\sup_{x \in B}\norm{\delta_r (x)} < e^\lambda$. Then~\eqref{eq_r1Bound} implies that $r_1(X) \geq r$, and thus $r_1(X) = \infty$.

Suppose now that $\EEE{e^{\lambda B(X)}} < \infty$ for some $\lambda > 0$. By Proposition~\ref{prop_contWithStrong}, the functions $\delta_r$ converge strongly to $\delta_0$ as $r \rightarrow 0$ and, in particular, uniformly on $B$. Thus there exists $r > 0$ such that $\sup_{x \in B}\norm{\delta_r (x)} < e^\lambda$. Then~\eqref{eq_r1Bound} implies that $r_1(X) \geq r > 0$ as desired.
\end{proof}

We demonstrate how to apply Theorem~\ref{thm_r1Pos} to random variables arising from signatures of geometric rough paths.

Let $V$ be a Banach space and $p \geq 1$. We note that for any $\x \in \Omega_p$, $\omega_\x(s,t) := \norm{\x}_{\pvar; [s,t]}^p$ defines a control for which~\eqref{eq_factorialBound} is satisfied.
Thus for all $k \geq 0$, the lift $\Lyons(\x) : \Delta_{[0,T]} \mapsto E$ satisfies
\[
\norm{\Lyons(\x)^k_{0,T}} \leq \frac{\omega(0,T)^{k/p}}{\beta_p(k/p)!}.
\]
We hence define
\[
K_p = \left\{ x \in E \mid \sup_{k \geq 0} \beta_p (k/p)! \norm{x^k} \leq 1 \right\}
\]
and observe that $\Lyons(\x)_{0,T} \in K_p$ for every $\x \in \Omega_p$ with $\norm{\x}_{\pvar;[0,T]} \leq 1$. Observe furthermore that $K_p$ is bounded and measurable in $E$.

For $\x \in \Omega_p$, define $k_p(\x) = K_p(\Lyons(\x)_{0,T})$, i.e., the minimum positive integer $k$ for which there exist $x_1,\ldots,x_k \in K_p$ such that $\Lyons(\x)_{0,T} = x_1\ldots x_k$.

We briefly recall the construction of the greedy sequence and function $N_{\kappa, [0,T], p}(\x)$ introduced in~\cite{CassLittererLyons13}. For $\kappa > 0$ define the sequence of times $\tau_0= 0$,
\[
\tau_{j+1} = \inf\{t > \tau_{j} \mid \omega_\x(\tau_{j},t) \geq \kappa\} \wedge T,
\]
so that $\omega_\x(\tau_{j},\tau_{j+1}) = \kappa$ for all $0 \leq j < N = N_{\kappa, [0,T], p}(\x) := \sup\{j \geq 0 \mid \tau_j < T\}$ and $\omega_\x(\tau_{N},\tau_{N+1}) \leq \kappa$ (see \cite{CassLittererLyons13} Definition~4.7,~\cite{FrizHairer14} p.158). Note that $k_p(\x) \leq N_{1, [0,T], p}(\x) + 1$.

\begin{remark}\label{remark_mForLift}
For any $p, q \geq 1$ and $\x \in \Omega_q$, note that the signature $\Lyons(\x)_{0,T}$ exists and so $k_p(\x)$ is meaningfully defined. Moreover, in case $q \leq p$, $\x$ can canonically be viewed as an element of $\Omega_p$ by its lift $S_{\floor p} \x \in \Omega_p$, and we have $\Lyons(\x)_{0,T} = S(S_{\floor p} \x)_{0,T}$.

However, if $q < \floor{p}$ and $N_{1, [0,T], p}(\x)$ and the greedy sequence $(\tau_j)_{j=1}^{\infty}$ are defined in terms of $\x$ (not its lift $\Lyons_{\floor p}\x$), then $N_{1, [0,T], p}(\x)$ does not yield a deterministic bound on $k_p(\x)$ since the individual signatures $\Lyons(\x)_{\tau_j,\tau_{j+1}}$ will in general fail to be elements of $K_p$.

To obtain a bound on $k_p(\x)$, one needs to consider $N_{1, [0,T], p}(S_{\floor p} \x)$ and $(\tau_j)_{j=1}^{\infty}$ defined in terms of $S_{\floor p} \x \in \Omega_p$. Then $\Lyons(\x)_{\tau_j,\tau_{j+1}} = \Lyons(\Lyons_{\floor p} \x)_{\tau_j,\tau_{j+1}} \in K_p$ for all $j=0,1,\ldots$, and so $k_p(\x) \leq N_{1, [0,T], p}(S_{\floor p} \x) + 1$.
\end{remark}

Let $\KK_p(V)$ be the family of $\Omega_p$-valued random variables $\X$ such that $\EEE{e^{\lambda k_p(\X)}} < \infty$ for all $\lambda > 0$.

\begin{cor}\label{cor_KKpRadius}
Let $V$ be a Banach space, $p \geq 1$ and $\X \in \KK_p(V)$. Then $\ExpSig\left[\Lyons(\X)_{0,T}\right]$ has an infinite radius of convergence.
\end{cor}

\begin{cor}\label{cor_KKpUnique}
Let $p \geq 1$ and $\X \in \KK_p(\R^d)$ such that $\X$ is $G\Omega_p(\R^d)$-valued. Then $\Lyons(\X)_{0,T}$ is the unique $G(\R^d)$-valued random variable whose expected signature is $\ExpSig\left[\Lyons(\X)_{0,T}\right]$.
\end{cor}

We now demonstrate two important examples of $G\Omega_p(\R^d)$-valued random variables in $\KK_p(\R^d)$.
Remark that a non-negative random variable $Z$ satisfies $\EEE{e^{\lambda Z}} < \infty$ for all $\lambda > 0$ whenever $Z^{\theta}$ has a Gaussian tail for some $\theta > 1/2$, i.e., $\PPP{Z^\theta > z} \leq C^{-1} e^{-Cz^2}$ for all $z > 0$ and a constant $C > 0$. In both of the following examples $[0,T]$ is a fixed time interval.

\begin{example}[Gaussian rough paths]\label{ex_GaussianRPs}
Recall that every centred continuous Gaussian process in $\R^d$ with independent components and covariance matrix of finite 2D $\rho$-variation, $\rho \in [1,2)$, admits a natural lift to a $G\Omega_p(\R^d)$-valued random variable $\X$ for any $p > 2\rho$ (\cite{FrizVictoir10} Theorem~15.33).

We recall here several results from~\cite{CassLittererLyons13}. Particularly, in the case that $\rho \in [1,3/2)$ it holds that $N_{1, [0,T], p}(\X)^{1/\rho}$ has a Gaussian tail. Moreover, in the special case that $\X$ is the natural lift of fractional Brownian motion with Hurst parameter $H > 1/4$ for some $p > H^{-1}$ (and indeed the lift exists for any $p > H^{-1}$), it holds that $N_{1, [0,T], p}(\X)^{1/2 + 1/p}$ has a Gaussian tail (\cite{CassLittererLyons13} Corollary~5.5 and Theorem~6.3,~\cite{FrizHairer14} Theorem~11.13).

Since $k_p(\X) \leq N_{1, [0,T], p}(\X) + 1$ as remarked before, it follows that $\X$ is a $G\Omega_p(\R^d)$-valued random variable in $\KK_p(\R^d)$ in both of the above cases.
\end{example}

\begin{example}[Markovian rough paths]\label{ex_MarkovRPs}
Consider $V = \R^d$, $n \geq 1$, and $\g = \g^n(\R^d)$, which for convenience we identify with the Lie group $G^n(\R^d)$ via the exponential map. Let $\X = \X^{a,x}$ be a Markovian rough path constructed from a Dirichlet form $\EE^a$ on $L^2(\g)$ for $a \in \Xi^{n, d}(\Lambda)$, $\Lambda \geq 1$, and starting point $\X^{a,x}_0 = x \in \g$ (taking the natural lift when $n=1$, see~\cite{FrizVictoir10} Chapter~16 for definitions). The sample paths of this process are almost surely geometric $p$-rough paths for any $p > 2$.

A recent result of Cass and Ogrodnik (\cite{CassOgrodnik15} Theorem~5.3) implies that $N_{1, [0,T], p}(\X)^{1-1/p}$
has a Gaussian tail for any $p > 2$ (moreover the constant determining the tail bounds depends only on $\Lambda, p, n, d$ and $T$). It follows that $\X$ is a $G\Omega_p(\R^d)$-valued random variable in $\KK_p(\R^d)$ for all $p > 2$.
\end{example}

\subsection{Analyticity}\label{subsec_analyticity}

In this section we investigate conditions under which the characteristic function is analytic. We apply these results to situations where the expected signature does not necessarily have an infinite radius of convergence.

\begin{definition}
Let $X$ be an $E$-valued random variable, $H$ a finite dimensional Hilbert space, and $M \in \LLL(V, \LLL(H))$. For $\lambda \in \C$, define $\phi_{X, M}(\lambda) = \EEE{(\lambda M)(X)}$ whenever $\norm{(\lambda M)(X)}$ is integrable.
\end{definition}

The above definition of $\phi_{X, M}$ does not introduce any new concept to the previously defined $\phi_X$ and simply makes the results in this section easier to state.

Recall that for a real random variable $X$, if $\EEE{|e^{\lambda X}|} < \infty$ for all $\lambda \in (-\varepsilon,\varepsilon)$, then $\phi_X(\lambda) = \EEE{e^{i\lambda X}}$ is well-defined and analytic on the strip $|\operatorname{Im}(z)| < \varepsilon$. This property is known as the propagation of regularity (and similar results hold for $C^{2k}$ regularity of $\phi_X$ on $\R$, see, e.g.,~\cite{Lukacs70}).

We start by showing that the analogue of this property is not in general true for $G$-valued random variables whenever $\dim(V) \geq 2$. The propagation of regularity for real (or equivalently $G(\R)$-valued) random variables relies crucially on commutativity in $E(\R)$, and we show how the lack of commutativity prevents the same phenomenon from occurring when $\dim(V) \geq 2$. Recall the radius of convergence $r_1(X)$ from Definition~\ref{def_expSig}.

\begin{example}\label{ex_noPropag}
Let $V$ be a normed space with $\dim(V) \geq 2$. We construct a $G$-valued random variable $X$ such that
\begin{enumerate}[label={(\arabic*)}]
\item \label{point_radPos} $r_1(X) > 0$, thus in particular, $\ExpSig(X)$ exists and $\phi_{X,M}$ is analytic in a neighbourhood of zero for all $M \in \A$,
\item \label{point_dense} there exists $M \in \LLL(V, \uu(\C^2))$ such that the set of $\lambda \in \C$ for which $|\lambda| > 1$ and $\EEE{\norm{(\lambda M)(X)}} = \infty$ forms a dense subset of $\{z \in \C \mid |z| > 1\}$, and
\item \label{point_nowhereDiff} $\phi_{X,M}$ is nowhere differentiable on $(1,\infty)$.
\end{enumerate}

Let $e_1, e_2$ be fixed linearly independent vectors in $V$ of unit length, $s$ a non-negative real random variable, and $N$ a random variable in $\N = \{0,1,2,\ldots\}$ independent of $s$. Define the $G$-valued random variable $X = \exp(s f_N)$, where $f_N = [e_1,[\ldots,[e_1,e_2]\ldots]$ with $e_1$ appearing $N$ times.

Suppose there exists $r_0 > 0$ such that that $\EEE{e^{r s}} < \infty$ for all $0 \leq r < r_0$. We claim this implies~\ref{point_radPos}. Indeed, remark that $\norm{f_n} \leq 2^n$, and thus $\norm{\delta_r X} \leq \exp(2^N r^N s)$. Denote $p_n = \PPP{N = n}$. It follows that for $r > 0$ sufficiently small 
\[
\EEE{\norm{\delta_r X}}
\leq \sum_{n \geq 0} p_n \EEE{\exp(2^n r^n s)}
< \infty,
\]
which implies $r_1(X) > 0$ as claimed.

Let $\su(2)$ be the special unitary Lie algebra of dimension $3$ with the standard basis $u_1,u_2,u_3$ satisfying $[u_1,u_2]=u_3$, $[u_2,u_3] = u_1$, $[u_3,u_1] = u_2$. Let $M : V \mapsto \su(2)$ defined by $e_i \mapsto u_i$ for $i=1,2$ and arbitrary otherwise.

Suppose moreover that $\EEE{e^{r s}} = \infty$ for all $r > r_0$ and that $N$ has unbounded support. We claim this implies~\ref{point_dense}. Indeed, let $v_n = M(f_n)$ (thus $v_0 = u_2, v_1 = u_3, v_2 = -u_2, v_3 = -u_3, v_4 = v_0, \ldots$). Denote $\lambda = re^{i\theta}$ for $r, \theta \geq 0$, so that $(\lambda M)(X) = \exp(s r^Ne^{iN\theta} v_N)$. We obtain
\begin{equation*}
\begin{aligned}
\EEE{\norm{(\lambda M)(X)}}
& = \sum_{n \geq 1} p_n \EEE{\norm{\exp(s r^n e^{in\theta} v_n)}} \\
& = \sum_{n \geq 1} p_n \EEE{\exp\left(\frac{1}{2}|s r^n\sin(n\theta)|\right)},
\end{aligned}
\end{equation*}
where the last equality follows since $i v_n$ is Hermitian with eigenvalues $\pm \frac{1}{2}$.

Let $D$ be any open subset of $\{z \in \C \mid |z| > 1\}$. We observe that there exists $n \geq 1$ sufficiently large and $r > 1, \theta > 0$, such that $p_n > 0, re^{i\theta} \in D$ and $\EEE{\exp\left(\frac{1}{2}|s r^n\sin(n\theta)|\right)} = \infty$. Thus~\ref{point_dense} holds as claimed.

Finally, we make specific choices for $s$ and $N$ to obtain~\ref{point_nowhereDiff}. Observe by Fubini's theorem that for all $r \geq 0$
\[
\EEE{(r M)(X)} = \sum_{n \geq 1} p_n \EEE{\exp(s r^n v_n)}.
\]
Suppose $p_n > 0$ only if $n = 4m$ for some integer $m$. Since
\[
\exp(t u_3)
= \left( \begin{array}{cc} e^{it/2} & 0  \\ 0 & e^{-it/2} \end{array} \right),
\]
it follows that 
\begin{equation*}
\EEE{(r M)(X)}
= \sum_{n \geq 0} p_{4n} \left( \begin{array}{cc} \exp\left(i r^{4n} s /2\right) & 0  \\ 0 & \exp\left(-i r^{4n} s/2\right) \end{array} \right).
\end{equation*}
It follows that $\phi_{X,M}$ has the same regularity at $r > 0$ as $\sum_{n \geq 0}p_{4n}\phi_s(r^{4n}/2)$, where $\phi_s$ is the characteristic function of $s$.

It is now easy to find $\phi_s$ and $p_n$ such that the above series defines a nowhere differentiable function on $(1,\infty)$. For example, let $\phi_s(\lambda) = (1-q)(1-qe^{i\lambda})^{-1}$ for any $0<q<1$, i.e., $s$ is geometrically distributed with parameter $1-q$, and let $p_n$ decay faster than any geometric sequence, i.e., for any $\alpha \in (0,1)$ there exists $n_\alpha$ such that $p_n < \alpha^n$ for all $n \geq n_\alpha$. The statement of~\ref{point_nowhereDiff} then follows by Dini's general construction of a nowhere differentiable function (\cite{Knopp18} p.24).
\end{example} 

\begin{remark}
The random variable $X$ constructed above is the exponential of a Lie polynomial of degree $N$. Thus when $V = \R^d$, $X$ is the signature of a random weakly geometric $N$-rough path (\cite{FrizVictoir10} Exercise~9.17), and thus of a random geometric $p$-rough path for $p > N$. Moreover, as the decay of $\norm{X^k}$ is exactly of the order $(k/N)!^{-1}$, there does not exist a fixed $p \geq 1$ such that $X$ is almost surely the signature of a random geometric $p$-rough path.

One can however approximate each sample of $X$ by the signature $\Lyons(\X)_{0,T}$ of a bounded variation path $\X_{0,\cdot} : [0,T] \mapsto \R^d$ in such as way that~\ref{point_radPos} and~\ref{point_dense} in Example~\ref{ex_noPropag} still hold for the $G(\R^d)$-valued random variable $\Lyons(\X)_{0,T}$ with the change that the stated $\lambda$ in~\ref{point_dense} will be dense in the annulus $\{z \in \C \mid 1 < |z| < R\}$ for any fixed $R > 1$ (where the random variable $\X$ depends on $R$).
\end{remark}

\begin{definition}\label{def_phi}
Let $V$ be a normed space. Denote by $\Phi(V)$ the set of $G$-valued random variables $X$ which satisfy
\begin{enumerate}[label={(P\arabic*)}]
\item \label{point_pos} $r_1(X) > 0$, and
\item \label{point_hol} $\phi_{X,M}$ is (weakly) analytic on $\R$ for all $M \in \A$.
\end{enumerate}
\end{definition}

The importance of the set $\Phi$ is that when $V = \R^d$ and $X,Y \in \Phi(\R^d)$ such that $\ExpSig(X) = \ExpSig(Y)$, we have $X \eqd Y$. To observe this, remark that for $V$ normed and $X$ an $E$-valued random variable with $r_1(X) =: \varepsilon > 0$, it follows by dominated convergence that $\EEE{M(X)} = \sum_{k\geq 0} M^{\otimes k}\EEE{X^k}$ whenever $\norm{M} < \varepsilon$. Hence for all $M \in \A$, $\phi_{X,M}(\lambda)$ is completely determined by $\ExpSig(X)$ whenever $|\lambda|$ is sufficiently small, and the claim follows by Corollary~\ref{cor_uniqueMeasure}.

Theorem~\ref{thm_charFuncAnalytic} is the main result of this section and provides a criterion to ensure that $X \in \Phi$.

\begin{theorem}\label{thm_charFuncAnalytic}
Let $V$ be a normed space and $X$ a $U$-valued random variable. Suppose there exists a bounded, measurable set $B \subset U$ such that $B(X)$ has an exponential tail, i.e., $\EEE{e^{\lambda B(X)}} < \infty$ for some $\lambda >0$. Then \ref{point_pos} and~\ref{point_hol} hold for $X$.
\end{theorem}

\begin{proof}[Proof of Theorem~\ref{thm_charFuncAnalytic},~\ref{point_pos}]
This follows immediately from Theorem~\ref{thm_r1Pos}.
\end{proof}

For the proof of~\ref{point_hol}, we require the following two lemmas.

\begin{lem}\label{lem_integrableFromM}
Let $F$ be a topological algebra, $A$ a normed algebra, $M\in\Hom(F,A)$, $B \subset F$ a bounded measurable set, and let $c = \sup_{x \in B}\norm{M(x)}$. Let $X$ be an $F$-valued random variable such that $\norm{M(X)}$ and $(c+\varepsilon)^{B(X)}$ are integrable for some $\varepsilon > 0$.

Then $\sup_{M' \in \UU}\norm{M'(X)}$ is an integrable random variable, where
\[
\UU = \{M' \in \Hom(F,A) \mid \sup_{x \in B}\norm{M(x)-M'(x)} < \varepsilon\}
\]
is an open subset of $\Hom(F,A)$ (under the strong topology).
\end{lem}

\begin{proof}
By definition of the strong topology, $\sup_{x \in B}\norm{\cdot(x)}$ is a semi-norm on $\LLL(F,A)$ and so $\UU$ is indeed an open subset.

Moreover $x \mapsto \sup_{M' \in \UU}\norm{M'(x)}$ is the supremum of a family of continuous functions, thus lower semi-continuous, and thus measurable. The claim now follows by a direct application of~\eqref{eq_prodBound}.
\end{proof}

For a bounded complex domain $D \subset \C$, denote by $H_b(\overline D)$ the space of continuous functions on $\overline D$ which are analytic on $D$. Recall that $H_b(\overline D)$ equipped with the uniform norm is a separable Banach space.

\begin{lem}\label{lem_weakHol}
Let $V$ be a normed space, $A$ a separable Banach algebra and $M \in \LLL(V,A)$. Let $X$ be an $E$-valued random variable. Assume that for a bounded domain $D \subset \C$, $\sup_{\lambda \in \overline D}\norm{(\lambda M)(X)}$ is an integrable random variable. Then for every $f \in A'$, the map $\lambda \mapsto \EEE{\gen{f,(\lambda M)(X)}}$ is in $H_b(\overline D)$.
\end{lem}

\begin{proof}
Let $f \in A'$. For $x \in E$ consider the map $\phi_{M,f}(x) : \lambda \mapsto \gen{f, (\lambda M)(x)}$, which is an entire function on $\C$.

We claim that the corresponding linear map $\phi_{M,f} : E \mapsto H(\C)$, where $H(\C)$ is the space of entire functions on $\C$, is bounded when we equip $H(\C)$ with the compact-open topology. Indeed, since $\lambda \mapsto (\lambda M)$ is a continuous map from $\C$ into $\Hom(E,A)$ by Proposition~\ref{prop_contWithStrong}, the collection of maps $(\lambda M)_{\lambda \in K}$ is strongly bounded in $\Hom(E,A)$ for any bounded set $K \subset \C$.
Thus for every bounded set $L \subset E$, it holds that
\[
\sup_{x \in L} \sup_{\lambda \in K} \norm{(\lambda M)(x)} < \infty.
\]
In particular, this implies that $\phi_{M, f}(L)$ is a bounded subset of $H(\C)$ for every bounded set $L \subset E$ as claimed.

Since $E$ is a Fr{\'e}chet space (hence bornological), it follows moreover that $\phi_{M,f}: E \mapsto H(\C)$ is continuous.
Hence $\phi_{M,f}(X)|_{\overline D}$ is a norm-integrable $H_b(\overline D)$-valued random variable and thus possesses a barycenter $h \in H_b(\overline D)$.

Let $\lambda \in \overline D$. Since the evaluation map $\gen{\cdot, \lambda} : x \mapsto x(\lambda)$ is in the continuous dual of $H_b(\overline D)$, it follows that
\[
h(\lambda) = \EEE{\gen{\phi_{M,f}(X), \lambda}} = \EEE{\gen{f, (\lambda M)(X)}} = \gen{f, \EEE{(\lambda M)(X)}},
\]
where the last equality follows since $(\lambda M)(X)$ is a norm-integrable $A$-valued random variable and is thus weakly integrable by the separability of $A$. As $h$ is in $H_b(\overline D)$, the conclusion follows.
\end{proof}

\begin{proof}[Proof of Theorem~\ref{thm_charFuncAnalytic},~\ref{point_hol}]
Let $M \in \A$. Since $\norm{M(g)} = 1$ for all $g \in U$, one obtains from Proposition~\ref{prop_contWithStrong} and Lemma~\ref{lem_integrableFromM} that there exists a domain $D$ containing $1 \in \C$ such that $\sup_{\lambda \in \overline D} \norm{(\lambda M)(X)}$ is an integrable random variable. The conclusion now follows by applying Lemma~\ref{lem_weakHol}.
\end{proof}

Following the discussion at the end of Section~\ref{subsec_momentsProb}, define
\[
N_p = K_p \cap U = \left\{ g \in U \mid \sup_{k \geq 0} \beta_p (k/p)! \norm{g^k} \leq 1 \right\}.
\]
We observe that $\Lyons(\x)_{0,T} \in N_p$ for every $\x \in G\Omega_p$ with $\norm{\x}_{\pvar;[0,T]} \leq 1$. As with $K_p$, $N_p$ is bounded and measurable in $E$.

For $\x \in G\Omega_p$, define $n_p(\x) = N_p(\Lyons(\x)_{0,T})$, i.e., the minimum positive integer $n$ for which there exist $g_1,\ldots,g_n \in N_p$ such that $\Lyons(\x)_{0,T} = g_1\ldots g_n$. Recall the functions $k_p$ and $N_{1, [0,T], p}$ from Section~\ref{subsec_momentsProb} and note that $k_p(\x) \leq n_p(\x) \leq N_{1, [0,T], p}(\x) + 1$.

\begin{remark}\label{remark_nForLift}
As in Remark~\ref{remark_mForLift}, we mention again that for $1 \leq q \leq p$, every $\x \in G\Omega_q$ is canonically defined as an element of $G\Omega_p$ via its lift $S_{\floor p} \x \in G\Omega_p$. However one cannot bound $n_p(\x)$ in terms $N_{1, [0,T], p}(\x)$ computed directly in terms of $\x$; instead one has $n_p(\x) \leq N_{1, [0,T], p}(S_{\floor p} \x) + 1$.
\end{remark}

Let $\NN_p(V)$ be the family of $G\Omega_p$-valued random variables $\X$ such that $n_p(\X)$ has an exponential tail. Note that if $\X \in \KK_p$ and is $G\Omega_p$-valued, then $\X \in \NN_p$.

\begin{cor}\label{cor_controlOnMp}
Let $V$ be a Banach space. Then for all $p \geq 1$ and $\X \in \NN_p$, the signature $\Lyons(\X)_{0,T}$ is a $U$-valued random variable satisfying~\ref{point_pos} and~\ref{point_hol}.
\end{cor}

In the finite dimensional setting, we obtain a result analogous to Corollary~\ref{cor_KKpUnique} but with weaker assumptions and a weaker conclusion.

\begin{cor}\label{cor_uniqueInPhi}
Let $p \geq 1$ and $\X \in \NN_p(\R^d)$. Then $\Lyons(\X)_{0,T} \in \Phi(\R^d)$. In particular, $\Lyons(\X)_{0,T}$ is the unique $G(\R^d)$-valued random variable in $\Phi(\R^d)$ whose expected signature is $\ExpSig\left[\Lyons(\X)_{0,T}\right]$.
\end{cor}

\begin{remark}
For a random variable $X \in \Phi(\R^d)$, we cannot exclude the possibility that there exists a $G(\R^d)$-valued random variable $Y$ (which might arise as the signature of a geometric rough path) such that $Y \notin \Phi(\R^d)$ and $\ExpSig(X) = \ExpSig(Y)$. Whether this is possible currently remains unknown.

However, note that Corollaries~\ref{cor_controlOnMp} and~\ref{cor_uniqueInPhi} apply for all $p \geq 1$.
Thus for any $p, q \geq 1$ and random geometric rough paths $\X \in \NN_p(\R^d)$ and $\Y \in G\Omega_q(\R^d)$, if $\Lyons(\X)_{0,T}$ and $\Lyons(\Y)_{0,T}$ are not equal in law (as $G(\R^d)$-valued random variables) and $\ExpSig(\Lyons(\X)_{0,T}) = \ExpSig(\Lyons(\Y)_{0,T})$, then the lift $S_{\floor{q'}}\Y$ cannot be in $\NN_{q'}(\R^d)$ for any $q' \geq q$.
\end{remark}

\begin{example}[Markovian rough paths stopped upon exiting a domain]
Recall the notation of Example~\ref{ex_MarkovRPs} and the result of Cass and Ogrodnik~\cite{CassOgrodnik15} that $N_{1, [0,1], p}(\X^{a,x})^{1-1/p}$ has a Gaussian tail for any $p > 2$. 

In this example we shall replace the interval $[0,1]$ by $[0,T]$, where $T$ is the first exit time of $\X^{a,x}$ from a suitable set. In particular, we shall show that $N_{\kappa, [0,T], p}(\X^{a,x})$ has an exponential tail and that this result is asymptotically sharp.

Throughout the example we fix $\Lambda \geq 1$ and $\g = \g^n(\R^d)$. We first give a slight extension of the support theorem~\cite{FrizVictoir10} Theorem~16.33 in the H{\"o}lder topology. Recall the Sobolev path space $W^{1,2}_x([s,t], \g)$ with starting point $x \in \g$.
In particular, recall that for all $\hh \in W^{1,2}_x([s,t], \g)$ and $\alpha \in [0,1/2]$
\[
\norm{\hh}_{\aHol;[s,t]} \leq (t-s)^{1/2-\alpha}\norm{\hh}_{W^{1,2};[s,t]}.
\]
For $\theta > 0$ consider the ball
\[
W_{\theta;x} := \{\hh \in W^{1,2}_x([0,1], \g) \mid \norm{\hh}_{W^{1,2};[0,1]} < \theta\}.
\]

\begin{lem}\label{lem_MarkovianSupport}
For any $\alpha \in [0,1/4)$, $\theta > 0$ and $c > 0$, there exists $\delta > 0$ such that
\[
\mathbb{P}^{a,x}\left[d_{\aHol;[0,1]}(\X, \hh) < c \right] > \delta
\]
for all $a \in \Xi^{n, d}(\Lambda)$, starting points $x \in \g$, and $\hh \in W_{\theta;x}$.
\end{lem}


The proof is essentially the same as that in~\cite{FrizVictoir10} and we defer it to the end of the example.
Recall now the greedy sequence $(\tau_j)_{j=1}^{\infty}$ associated with $N_{\kappa, [0,T], p}(\X^{a,x})$. For ease of notation, we shall not stop $\tau_j$ at $T$ for $j > N := N_{\kappa, [0,T], p}(\X^{a,x})$ (i.e., we do not necessarily have $\tau_{N+1} = T$). Note this causes no confusion since $\X^{a,x}_t$ is defined for all times $t \geq 0$ as a diffusion on $\g$.

Consider first $\X^{a,x} : [0,1] \mapsto \g$. Taking $\hh \equiv x$ the trivial path, Lemma~\ref{lem_MarkovianSupport} implies that for any $p > 4$ and $\kappa > 0$, there exists $\delta > 0$ such that
\[
\inf_{x \in \g}\mathbb{P}^{a,x}\left[\norm{\X}_{\pHol;[0,1]} < \kappa\right] \geq \delta.
\]
It follows that
\[
\inf_{x \in \g} \mathbb{P}^{a,x}\left[\tau_1 > 1 \right] \geq \inf_{x \in \g}\mathbb{P}^{a,x}\left[\norm{\X}_{\pHol;[0,1]} < \kappa\right] \geq \delta,
\]
so by the (strong) Markov property of $\X^{a,x}$ and properties of conditional expectation
\[
\mathbb{P}^{a,x}\left[ N_{\kappa, [0,1], p}(\X) \geq k \right] = \mathbb{P}^{a,x}\left[\tau_k < 1\right] \leq (1-\delta)^k.
\]
That is, $N_{\kappa, [0,1], p}(\X^{a,x})$ has an exponential tail (moreover $\delta$ does not depend on $x \in \g$ or $a \in \Xi^{n,d}(\Lambda)$).

While this argument yields a strictly weaker asymptotic bound than that in~\cite{CassOgrodnik15}, the advantage is that by choosing appropriate $\hh$ in Lemma~\ref{lem_MarkovianSupport}, a very similar argument gives upper and lower bounds on the tail of $N_{\kappa, [0,T], p}(\X^{a,x})$, where $T$ is now the first exit time of $\X^{a,x}$ from a suitable open set.
We first show the lower bound.

Recall that $\g$ is equipped with the (left-invariant) metric $d$ induced by the Carnot-Carath{\'e}odory norm (or any other symmetric sub-additive homogeneous norm). For any $r > 0$ and $x \in \g$, define $B_r(x) = \{y \in \g\mid d(x,y) \leq r\}$.

\begin{proposition}\label{prop_MarkovExitLower}
Let $p > 4$, $\kappa, r > 0$. Define $T = \inf\{t > 0 \mid \X^{a,x}_t \notin B_{r}(x)\}$ the first exit time of $\X^{a,x}$ from $B_{r}(x)$. Then there exists $\delta > 0$ such that
\[
\mathbb{P}^{a,x}\left[N_{\kappa, [0,T], p}(\X) \geq k\right] \geq \delta^k
\]
for all $a \in \Xi^{n, d}(\Lambda)$ and $x \in \g$.
\end{proposition}

\begin{proof}
Let $\theta > 0$ sufficiently large such for all $x \in \g$ and $y \in B_{r/2}(x)$ there exists $\hh^y \in W_{\theta; y}$ such that $\hh^y_1 = x$, $\hh^y_t \in B_{r/2}(x)$ for all $t \in [0,1]$, and $\norm{\hh^y}_{\pvar;[0,1]} > \kappa + r/2$ (for example, take $t \mapsto \hh^y_t$ as a geodesic from $y$ to $x$ on $[0,1/2]$ and then $\hh^y_{t} = x\hh_{2t-1}$ for $t \in [1/2,1]$ for a fixed $\hh \in W^{1,2}_0([0,1],\g)$ with $\hh_1 = 0$, $\hh_t \in B_{r/2}(0)$ for all $t \in [0,1]$, and $\norm{\hh}_{\pvar;[0,1]} > \kappa + r/2$). 

Since $\norm{\X}_{\pvar;[0,1]} \geq \norm{\hh^y}_{\pvar;[0,1]} - d_{\pHol;[0,1]}(\X,\hh^y)$, we have for all $x \in \g$, $y \in B_{r/2}(x)$ and $a \in \Xi^{n, d}(\Lambda)$ 
\begin{multline*}
\mathbb{P}^{a,y}\left[\X_t \in B_r(x) \textnormal{ for all } t \in [0,1] , \norm{\X}_{\pvar;[0,1]} > \kappa, \X_1 \in B_{r/2}(x)\right] \\
\geq \mathbb{P}^{a,y}\left[d_{\pHol}(\X, \hh^y) < r/2 \right].
\end{multline*}
Applying Lemma~\ref{lem_MarkovianSupport} with $c = r/2$ and $\alpha = 1/p$, along with the (weak) Markov property and conditional expectation, concludes the proof.
\end{proof}

\begin{remark}
Note that Proposition~\ref{prop_MarkovExitLower} deals only with the quantity $N_{\kappa, [0,T], p}(\X^{a,x})$ and does not provide a lower bound on the tail of $n_p(\X^{a,x}_{[0,T]})$. In particular, one cannot conclude that $\ExpSig[\Lyons(\X^{a,x})_{0,T}]$ does not have an infinite radius of convergence.
\end{remark}

We now show an upper bound on the tail of $N_{\kappa, [0,T], p}(\X^{a,x})$ which will imply that $S(\X^{a,x})_{0,T} \in \Phi(\R^d)$ (see however Remark~\ref{remark_diffusionLift}). For a subset $D \subset \g$, consider the following property:
\begin{equation}\label{eq_domainCond}
\textnormal{There exist $r, c > 0$ such that $\sup_{h \in B_r(0)}\inf_{y \in D} d(xh,y) > c$ for all $x \in D$}.
\end{equation}

\begin{remark}\label{remark_domainTrunc}
For $1 \leq k \leq n$, let $\pi^k : \g^n(\R^d) \mapsto \g^k(\R^d)$ denote the projection. Then whenever the image $\pi^k(D)$ satisfies~\eqref{eq_domainCond} for some $1 \leq k \leq n$ (for the respective metric on $\g^k(\R^d)$), then so does $D$ (with a different choice of $r,c$).

Indeed, on the one hand $d(x,y) \geq d(\pi^k(x), \pi^k(y))$ for all $x,y \in \g^n(\R^d)$. On the other hand, for every $r > 0$, there exists $R > 0$ such that $B_r(0) \subset \pi^k(B_R(0)) \subset \g^k(\R^d)$. The conclusion readily follows since $\pi^k$ is a group homomorphism.
\end{remark}

\begin{proposition}\label{prop_MarkovExitUpper}
Let $p > 4$, $\kappa > 0$, and $D \subset \g$ be an open set satisfying~\eqref{eq_domainCond} for some $r,c > 0$. Define $T = \inf\{t > 0 \mid \X^{a,x}_t \notin D\}$ the first exit time of $\X^{a,x}$ from $D$. Then there exists $\delta > 0$ such that
\[
\mathbb{P}^{a,x}\left[N_{\kappa, [0,T], p}(\X) \geq k\right] \leq (1-\delta)^k
\]
for all $a \in \Xi^{n, d}(\Lambda)$ and $x \in D$.
\end{proposition}

\begin{proof}
Let $\theta > 0$ be sufficiently large such that for every $h \in B_r(0)$ there exists $\hh \in W_{\theta;0}$ such that $\hh_1 = h$.
Note that it suffices to prove the statement for any fixed $\kappa > 0$. In particular, we may assume that $\kappa > \theta + c$.

It follows that to every point $x \in D$, we can assign $h^x \in B_r(x)$ and $\hh^x \in W_{\theta;x}$ such that $\inf_{y \in D} d(h^x,y) > c$ and $\hh^x_1 = h^x$. Then for all $a \in \Xi^{n, d}(\Lambda)$ and $x \in D$
\begin{align*}
\mathbb{P}^{a, x}\left[\tau_1 > 1 \geq T\right]
&\geq \mathbb{P}^{a,x}\left[\norm{\X}_{\pvar;[0,1]} < \theta + c, \X_1 \notin D \right] \\
&\geq \mathbb{P}^{a,x}\left[\norm{\X}_{\pHol;[0,1]} < \theta + c, d(\X_1, h^x) < c\right] \\
&\geq \mathbb{P}^{a,x}\left[d_{\pHol;[0,1]}(\X, \hh^x) < c\right].
\end{align*}
Applying Lemma~\ref{lem_MarkovianSupport} with $\alpha = 1/p$, along with the (strong) Markov property and conditional expectation, concludes the proof.
\end{proof}

\begin{remark}\label{remark_diffusionLift}
The diffusion $\X^{a,x}$ is constructed on the space $\g^n = \g^n(\R^d)$ (or equivalently on $G^n(\R^d)$), and Proposition~\ref{prop_MarkovExitUpper} gives an exponential bound on the tail of $N_{1, [0,T], p}(\X^{a,x})$ computed in terms of $\X^{a,x}$ for any $p > 4$. Fixing $4 < p < 5$, Corollary~\ref{cor_uniqueInPhi} thus implies that $\Lyons(\X^{a,x})_{0,T} \in \Phi(\R^d)$ for $n \geq 4$.

One could extend this to the case $n=2$ or $3$ (recall for $n=1$ we consider the diffusion $\X^{a\circ \pi^1, x}$ on $\g^2$) if the analogue of Lemma~\ref{lem_MarkovianSupport} were true for all $\alpha \in [0,1/2)$. However such a support theorem is currently unknown.

Nonetheless, in light of Remarks~\ref{remark_mForLift} and~\ref{remark_nForLift}, for $n = 2$ or $3$ we can still show that $\Lyons(\X^{a,x})_{0,T} \in \Phi(\R^d)$ by showing that $N_{1, [0,T], p}(\Lyons_{4} \X^{a,x})$ has an exponential tail.

To show this, note we can apply Proposition~\ref{prop_MarkovExitUpper} to the diffusion $\X^{a\circ \pi^n, y}$ on $\g^4$ and the open set $(\pi^n)^{-1}(D) \subset \g^4$ (which indeed satisfies~\eqref{eq_domainCond} due to Remark~\ref{remark_domainTrunc}). We thus obtain that $N_{1, [0,\widetilde T], p}(\X^{a\circ\pi^n, y})$ has an exponential tail, where $\widetilde T$ is the first exit time of $\pi^n \X^{a \circ\pi^n, y}$ from $D$.

To conclude that $N_{1, [0,T], p}(\Lyons_{4} \X^{a,x})$ has an exponential tail, it suffices to show that $\Y^y_\cdot := y * S_4 \X^{a, \pi^n y}_{0,\cdot}$ is equal in law to $\X^{a\circ \pi^n, y}_{\cdot}$ for all $y \in \g^4$ as processes on $\g^4$ ($*$ denoting group multiplication in $\g^4$).
This follows by a similar argument as~\cite{FrizVictoir08} Section~6: observe that the Markov process $\Y^y_t$ is the solution of an RDE with starting point $y \in \g^4$ and driven by $\pi^2(\X^{a, \pi^n y}_t)$ (which is non-Markov in general) along the (unbounded) canonical left-invariant vector fields $U_1,\ldots, U_d$ on $\g^4$. Denoting by $P_t$ the semi-group on $C_b(\g^4)$ of $\Y^y_t$, it suffices to show that $\lim_{t \rightarrow 0} \gen{t^{-1}(f - P_t f), g}_{L^2(\g^4)} = \EE^{a\circ\pi^n}(f,g)$ for all $f,g \in C^\infty_c(\g^4)$.

Consider $f,g \in C^\infty_c(\g^4)$ with support in $B_R(0) \subset \g^4$ and fix smooth vector fields $U_i^R$ which agree with $U_i$ on $B_{2R}(0)$ and have compact support. Let $\Y^{R,y}_t$ denote the RDE driven by $\pi_2(\X^{a,\pi^n(y)}_t)$ along $U_i^R$ starting at $\Y^{R,y}_0 = y$.
For all $y \in B_R(0)$ and $t \in [0,1]$, we have $\Y^{R,y}_t = \Y^y_t$ whenever $\Y^y_s \in B_{2R}(0)$ for all $s \in [0,t]$, whilst the probability that $\Y^y_s$ leaves $B_{2R}(0)$ in $[0,t]$ is bounded above by $C^{-1}\exp(-C t^{-2/p})$ for any $2< p < 3$ and some $C = C(R,p)$ (which follows from Fernique estimates on $\norm{\X^{a,x}}_{\pHol;[0,1]}$).

Defining $P_t^Rf(y) := \EEE{f(\Y^{R,y}_t)}$, it follows readily that $\lim_{t \rightarrow 0}\gen{t^{-1}(f - P_t f), g} = \lim_{t \rightarrow 0}\gen{t^{-1}(f - P_t^R f), g}$. Finally, the latter limit is now seen equal to $\EE^{a\circ\pi^n}(f,g)$ following~\cite{FrizVictoir08} Lemmas~26,~27 and the proof of Proposition~28 (note that one readily extends Lemma~27 to diffusions on $\g^n$ for $n > 2$, cf.~\cite{FrizVictoir10} Proposition~16.20).
\end{remark}

\begin{proof}[Proof of Lemma~\ref{lem_MarkovianSupport}]
We mimic the proofs of~\cite{FrizVictoir10} Lemma~16.32 and Theorem~16.33 while keeping track of constants.

For $\alpha \in [0,1/4)$, $\hh \in W^{1,2}_x([0,1], \g)$ and $\varepsilon > 0$ define the set
\[
B^{\hh}_{\varepsilon; \alpha} = \{\x \in C^{\aHol}_x([0,1],\g) \mid \norm{\x}_{\aHol} \leq 2\norm{\hh}_{\aHol} + 1, d_{\infty}(\x,\hh) \leq \varepsilon\}.
\]
We claim that for all $\alpha \in [0,1/4)$ and $\varepsilon > 0$, there exists $\delta > 0$ such that
\[
\mathbb{P}^{a,x}\left[\X \in B_{\varepsilon; \alpha}^{\hh} \right] > \delta
\]
for all $a \in \Xi^{n, d}(\Lambda)$, $x \in \g$, and $\hh \in W_{\theta;x}$.

Indeed, we follow the proof~\cite{FrizVictoir10} Lemma~16.32 (we also mention here that, directly as stated,~\cite{FrizVictoir10} Lemma~16.32 contains the minor error that it fails to hold for the trivial path $\norm{\hh}_{\aHol} = 0$; this is readily fixed by modifying the definition of their $B^{\hh}_\varepsilon$ to our definition above; moreover the proof of~\cite{FrizVictoir10} Theorem~16.33 then goes through unchanged).

Using Step 1 of the proof of~\cite{FrizVictoir10} Lemma~16.32, we obtain that for any $\beta \in (\alpha, 1/2)$, $\mathbb{P}^{a,x}\left[\X \in B^\hh_{\varepsilon; \alpha}\right] \geq \Delta_1 - \Delta_2$,
where $\Delta_1 = \mathbb{P}^{a,x}\left[d_\infty(\X, \hh) \leq \varepsilon\right]$
and
\[
\Delta_2 = \mathbb{P}^{a,x}\left[\norm{\X}_{\betaHol} > (\norm{\hh}_{\aHol} + 1)^{\beta/\alpha}(2\varepsilon)^{1-\beta/\alpha}\right].
\]
The claim will follow once we show that $\Delta_2/\Delta_1 \rightarrow 0$ as $\varepsilon \rightarrow 0$ uniformly over $a \in \Xi^{n, d}(\Lambda)$, $x \in \g$, and $\hh \in W_{\theta;x}$.

By~\cite{FrizVictoir10} Theorem~E.21, we have $\log(\Delta_1) \geq -c_1\varepsilon^{-2}$ where $c_1 = C (1+\norm{\hh}_{W^{1,2}})^2$ and $C$ is a constant depending only on the doubling and Poincar{\'e} constants of $\EE^a$, which in turn depend only on $\Lambda, n$ and $d$ (\cite{FrizVictoir10} Proposition~16.5 and Theorem~E.8).

On the other hand, the Fernique estimate in~\cite{FrizVictoir10} Corollary~16.12 implies that
\[
\log(\Delta_2) \leq -c_2(\norm{\hh}_{\aHol} + 1)^{2\beta/\alpha} \varepsilon^{2-2\beta/\alpha} \leq -c_2\varepsilon^{2-2\beta/\alpha},
\]
where $c_2$ depends only on $\beta$ and $\Lambda$. So for fixed $\alpha \in [0,1/4)$, choose any $\beta \in (2\alpha,1/2)$. Since $2-2\beta/\alpha < -2$, we see $\Delta_2/\Delta_1 \rightarrow 0$ as $\varepsilon \rightarrow 0$ uniformly over the desired variables, which proves the claim.

To conclude, we follow the proof of~\cite{FrizVictoir10} Theorem~16.33. By the $d_0/d_\infty$ estimate on $\g$ (\cite{FrizVictoir10} Proposition~8.15),
\[
d_0(\x,\hh) \leq Cd_\infty(\x,\hh) + Cd_\infty(\x,\hh)^{1/n}(\norm{\x}_\infty + \norm{\hh}_\infty)^{1-1/n},
\]
where $C = C(n,d)$, and so by interpolation (\cite{FrizVictoir10} Lemma~8.16) we have for all $x\in \g$, $\x, \hh \in C^{\aHol}_x([0,1],\g)$ and $0 \leq \alpha' < \alpha < 1/4$ that 
\begin{align*}
d_{\aprimeHol}(\x,\hh)
\leq& (\norm{\x}_{\aHol} + \norm{\hh}_{\aHol})^{\alpha'/\alpha} d_0(\x,\hh)^{1-\alpha'/\alpha} \\
\leq& C^{1-\alpha'/\alpha}(\norm{\x}_{\aHol} + \norm{\hh}_{\aHol})^{\alpha'/\alpha} \\
& \left[d_\infty(\x,\hh) + d_\infty(\x,\hh)^{1/n}(\norm{\x}_\infty + \norm{\hh}_\infty)^{1-1/n} \right]^{1-\alpha'/\alpha}.
\end{align*}
Since $\norm{\hh}_{\infty} \leq \norm{\hh}_{\aHol} \leq \norm{\hh}_{W^{1,2}}$, it follows for all $\hh \in W_{\theta;x}$ and $\x \in B^{\hh}_{\varepsilon; \alpha}$ that
\[
d_{\aprimeHol}(\x,\hh) \leq c_3(\varepsilon + \varepsilon^{1/n})^{1-\alpha'/\alpha},
\]
where $c_3$ depends only on $n, d, \alpha, \alpha'$ and $\theta$.

Choosing $\varepsilon > 0$ so that $c_3(\varepsilon + \varepsilon^{1/n})^{1-\alpha'/\alpha} < c$, it follows that there exists $\delta > 0$ such that
\[
\mathbb{P}^{a,x}\left[d_{\aprimeHol}(\X, \hh) < c \right] \geq \mathbb{P}^{a,x}\left[\X \in B^{\hh}_{\varepsilon;\alpha} < c \right] > \delta
\]
for all $a \in \Xi^{n, d}(\Lambda)$, $x \in \g$, and $\hh \in W_{\theta;x}$, which concludes the proof.
\end{proof}
\end{example}

\subsection{Convergence of measures}\label{subsec_convMeas}

We conclude the paper with a result analogous to the method of moments for weak convergence of $G(\R^d)$-valued random variables. We work first with a slightly general notion of coproduct spaces as this is the only structure of $E$ which we require.

\begin{definition}
A \emph{coproduct space} $(F, \Delta)$ is a locally convex space $F$ and a continuous linear map $\Delta : F \mapsto F^{\widehat\otimes 2}$, with the additional property that $G(F) := \{g \in F \mid \Delta(g) = g\otimes g, g \neq 0\}$ is closed in $F$.
Let $\PP_r(F)$ be the set of (weakly) integrable probability measures on $G(F)$.
\end{definition}

The extra condition that $G(F)$ is closed in $F$ will only arise in Lemma~\ref{lem_weakConv} to ensure that $G(F)$ is Polish whenever $F$ is. Remark that~\eqref{eq_squareBound} remains true for any coproduct space $(F, \Delta)$, $f \in F'$, and $\mu \in \PP(F)$ with support on $G(F)$.

\begin{lem}\label{lem_qDomp}
Let $(F, \Delta)$ be a nuclear coproduct space and $\gamma$ a semi-norm on $F$. There exists a semi-norm $\xi$ on $F$ such that $\mu(\gamma) \leq \sqrt{\xi(\mu^*)}$ for all $\mu \in \PP_r(F)$.
\end{lem}

\begin{proof}
Let $\zeta$ be a semi-norm on $F$ such that the canonical map $\widehat F_\zeta \mapsto \widehat F_\gamma$ is nuclear.
Increasing $\zeta$ by a scalar multiple if necessary, it follows that there exist $(f_n)_{n \geq 1} \in F'$ such that $\sum_{n \geq 1}\zeta(f_n) \leq 1$ and $\gamma \leq \sum_{n \geq 1}|f_n|$. The conclusion then follows from~\eqref{eq_squareBound} for any semi-norm $\xi$ on $F$ such that $\xi \geq (\zeta^{\otimes 2}) \circ \Delta$.
\end{proof}

\begin{lem}\label{lem_muCompact}
Let $(F, \Delta)$ be a Fr{\'e}chet nuclear coproduct space. Let $R \subseteq \PP_r(F)$ be a family of probability measures on $G(F)$ such that $(\mu^*)_{\mu \in R}$ is bounded. Then $R$ is uniformly tight.
\end{lem}

\begin{proof}
Let $(\gamma_n)_{n \geq 1}$ be a defining non-decreasing sequence of semi-norms on $F$. By Lemma~\ref{lem_qDomp} there exists a sequence of semi-norms $(\xi_n)_{n \geq 1}$ on $F$ such that $\mu(\gamma_n) \leq \sqrt{\xi_n(\mu^*)}$ for all $\mu \in \PP_r(F)$. Since $(\mu^*)_{\mu \in R}$ is bounded, $\sup_{\mu \in R} \xi_n(\mu^*) < \infty$ for every $n \geq 1$. 

Let $B_n = \{x \in F \mid \gamma_n(x) < 1\}$. For any sequence of positive reals $(\lambda_n)_{n \geq 1}$, the set $K = \bigcap_{n \geq 1} \lambda_n B_n$ is bounded in $H$ and thus relatively compact (\cite{Treves67} p.520).
For all $\mu \in \PP_r(F)$ we have that
\[
\mu(K^c)
\leq \sum_{n \geq 1}\mu(\{x \mid \gamma_n(x)
\geq \lambda_n\})
\leq \sum_{n \geq 1} \lambda_n^{-1} \mu(\gamma_n)
\leq \sum_{n \geq 1} \lambda_n^{-1} \sqrt{\xi_n(\mu^*)}.
\]
Taking $\lambda_n$ sufficiently large, it follows that $\sup_{\mu \in R} \mu(K^c)$ can be made arbitrarily small.
\end{proof}

\begin{lem}\label{lem_weakConv}
Let $(F, \Delta)$ be a Fr{\'e}chet nuclear coproduct space and let $(\mu_n)_{n \geq 1}$ be a sequence of measures in $\PP_r(F)$ such that $\mu^*_n \rightarrow x$ weakly for some $x \in F$. Then there exists $\mu \in \PP_r(F)$ and a subsequence $(n(k))_{k \geq 1}$ such that $\mu_{n(k)} \convd \mu$ and $x = \mu^*$.
\end{lem}

\begin{proof}
Recall that a Fr{\'e}chet Montel space (thus in particular a Fr{\'e}chet nuclear space) is always separable (\cite{Schaefer71} p.195), and hence Polish. As a closed subset of $F$, $G(F)$ is also Polish.

The sequence $(\mu_n^*)_{n \geq 1}$ is bounded (\cite{Rudin73} Theorem~3.18) thus there exists a convergent subsequence $\mu_{n(k)} \rightarrow \mu$ for some probability measure $\mu$ on $G(F)$ by Lemma~\ref{lem_muCompact}.

Let $f \in F'$. Since $\sup_{n \geq 1} \mu_n(f^2) = \sup_{n \geq 1} (f^{\otimes 2})(\Delta\mu_n^*) < \infty$, the sequence of image measures $(\mu_n f^{-1})_{n \geq 1}$ on $\R$ is uniformly integrable.
It follows that $f$ is $\mu$-integrable and $f(\mu_n^*) = \mu_n(f) \rightarrow \mu(f)$ (\cite{Bogachev07} Lemma~8.4.3). Thus $x = \mu^*$ and $\mu \in \PP_r(F)$.
\end{proof}

Recall that $E$ is Fr{\'e}chet and nuclear whenever $V$ is. The following is now a consequence of Lemma~\ref{lem_weakConv} and Proposition~\ref{prop_infRadUnique}.

\begin{theorem}\label{thm_methodMoments}
Let $(X_n)_{n \geq 1}$ be a sequence of $G(\R^d)$-valued random variables such that $\EEE{X_n} \in E(\R^d)$ exists (i.e., $r_2(X_n) = \infty$) for all $n \geq 1$. Suppose that $\EEE{X_n}$ converges to some $x \in E(\R^d)$ in the weak topology of $E(\R^d)$. Then there exists a unique integrable $G(\R^d)$-valued random variable $X$ in such that $X_n \convd X$ and $x = \EEE{X}$.
\end{theorem}

\begin{remark}
We remark that $F := P(\R) = \prod_{k \geq 0} (\R)^{\otimes k}$ is also a Fr{\'e}chet nuclear coproduct space under the product topology. Moreover the exponential map $\exp : \R \mapsto G(\R) = G(F)$ is a homeomorphism. One may then directly apply Lemma~\ref{lem_weakConv} to obtain a proof of the classical method of moments for real random variables: if $\mu_n$ are probability measures on $\R$ with finite moments $(m_n(j))_{j \geq 1}$ such that $\lim_{n \rightarrow \infty} m_n(j) = m(j)$ for every $j \geq 1$, then $(m(j))_{j \geq 1}$ are the moments of a probability measure $\mu$ on $\R$ for which $\mu_{n(k)} \convd \mu$ along a subsequence $(n(k))_{k \geq 1}$ (if $\mu$ is moment-determined then in fact $\mu_n \convd \mu$).
\end{remark}

\bibliographystyle{plain}
\bibliography{ChevyrevLyonsRefs-17-05-17PR}

\end{document}